\documentclass[11pt,reqno]{amsart}
\usepackage{enumerate}
\usepackage{url}
\usepackage{cite}
\usepackage{comment}
\usepackage{graphicx}
\usepackage{upgreek}
\usepackage{tikz}
\usetikzlibrary{arrows,shapes,trees,backgrounds}
\usepackage{fancyhdr}
\usepackage{amsfonts, amsmath, amssymb, amscd, amsthm, bm, cancel}
\usepackage[
linktocpage=true,colorlinks,citecolor=magenta,linkcolor=blue,urlcolor=magenta]{hyperref}
\usepackage[margin=1.1in]{geometry}

\newtheorem{prop}{Proposition}[section]
\newtheorem{thm}[prop]{Theorem}
\newtheorem{lem}[prop]{Lemma}
\newtheorem{cor}[prop]{Corollary}
\newtheorem{defn}[prop]{Definition}

\theoremstyle{definition}
\newtheorem*{ack}{Acknowledgments}
\theoremstyle{remark}
\newtheorem{rem}[prop]{Remark}

\numberwithin{equation}{section}
\allowdisplaybreaks

\begin{document}

\title[Volume preserving Gauss curvature flow]{Volume preserving Gauss curvature flow of convex hypersurfaces in the hyperbolic space}

\author[Y. Wei]{Yong Wei}
\address{School of Mathematical Sciences, University of Science and Technology of China, Hefei 230026, P.R. China}
\email{\href{mailto:yongwei@ustc.edu.cn}{yongwei@ustc.edu.cn}}
\author[B. Yang]{Bo Yang}
\address{Institute of Mathematics, Academy of Mathematics and Systems Sciences, Chinese Academy of Sciences,
Beijing, 100190, P. R. China}
\email{\href{mailto:boyang16@amss.ac.cn}{boyang16@amss.ac.cn}}
\author[T. Zhou]{Tailong Zhou}
\address{School of Mathematical Sciences, University of Science and Technology of China, Hefei 230026, P.R. China}
\email{\href{mailto:ztl20@ustc.edu.cn}{ztl20@ustc.edu.cn}}

\begin{abstract}
We consider the volume preserving flow of smooth, closed and convex hypersurfaces in the hyperbolic space $\mathbb{H}^{n+1} (n\geq 2)$ with the speed given by arbitrary positive power $\alpha$ of the Gauss curvature. We prove that if the initial hypersurface is convex, then the smooth solution of the flow remains convex and exists for all positive time $t\in [0,\infty)$. Moreover, we apply a result of Kohlmann which characterises the geodesic ball using the hyperbolic curvature measures and an argument of Alexandrov reflection to prove that the flow converges to a geodesic sphere exponentially in the smooth topology. This can be viewed as the first result for non-local type volume preserving curvature flows for hypersurfaces in the hyperbolic space with only convexity required on the initial data.
\end{abstract}

\maketitle

\tableofcontents


\section{Introduction}
Let $X_0:M^n\to\mathbb{H}^{n+1} (n\geq 2)$ be a smooth embedding such that $M_0=X_0(M)$ is a closed and convex hypersurface in the hyperbolic space $\mathbb{H}^{n+1}$. We consider the smooth family of embeddings $X:M^n\times [0,T)\to \mathbb{H}^{n+1}$ satisfying
\begin{equation}\label{flow-VMCF}
	\left\{\begin{aligned}
		\frac{\partial}{\partial t}X(x,t)=&~(\phi(t)-K^{\alpha})\nu(x,t),\\
		X(\cdot,0)=&~X_0(\cdot),
	\end{aligned}\right.
\end{equation}
where $\alpha>0$, $\nu$ is the unit outer normal of $M_t=X(M,t)$, $K$ is the Gauss curvature of $M_t$ and
\begin{equation}\label{eqphi}
	\phi(t)=\frac{1}{|M_t|}\int_{M_t}K^{\alpha}\,\mathrm{d}\mu_t
\end{equation}
such that the domain $\Omega_t$ enclosed by $M_t$ has a fixed volume $|\Omega_t|=|\Omega_0|$ along the flow \eqref{flow-VMCF}.

\begin{defn}
Let $M^n$ be a smooth closed hypersurface in the hyperbolic space $\mathbb{H}^{n+1}$. Denote the principal curvatures of $M$ by $\kappa=(\kappa_1,\cdots,\kappa_n)$.

(i). $M$ is said to be horospherically convex (or simply called $h$-convex),  if its principal curvatures satisfy $\kappa_i\geq 1$ for all $i=1,\cdots,n$ everywhere on $M$. Equivalently, $M$ is h-convex if for any point $p\in M$ there exists a horosphere enclosing $M$ and touching $M$ at $p$. We also say that a bounded domain $\Omega\subset \mathbb{H}^{n+1}$ is $h$-convex if its boundary $\partial\Omega$ is $h$-convex.

(ii). $M$ is said to be positively curved (or called having positive sectional curvatures), if its principal curvatures satisfy
      \begin{equation*}
       \kappa_i\kappa_j>1,\qquad \forall~i\neq j
      \end{equation*}
       everywhere on $M$.

(iii). $M$ is said to be convex if its principal curvatures satisfy $\kappa_i>0$ for all $i=1,\cdots,n$  everywhere on $M$. We also say that a bounded domain $\Omega\subset \mathbb{H}^{n+1}$ is convex if its boundary $\partial\Omega$ is convex.
\end{defn}
 As the main result of this paper, we prove the following convergence result for the flow \eqref{flow-VMCF} with convex initial hypersurfaces:
\begin{thm}\label{theo}
Let $X_0:M^n\to \mathbb{H}^{n+1} (n\geq 2)$ be a smooth embedding such that $M_0=X_0(M)$ is a closed and convex hypersurface in $\mathbb{H}^{n+1}$. Then for any $\alpha>0$, the volume preserving flow \eqref{flow-VMCF} has a unique smooth convex solution $M_t$ for all time $t\in[0,\infty)$, and the solution $M_t$ converges smoothly and exponentially as $t\to \infty$ to a geodesic sphere of radius $\rho_{\infty}$ which encloses the same volume as $M_0$.
\end{thm}

The volume preserving mean curvature flow
\begin{equation}\label{eqH}
	\frac{\partial}{\partial t}X(x,t)=(\phi(t)-H)\nu(x,t)
\end{equation}
was introduced by Huisken \cite{Hui87} in 1987 for convex hypersurfaces in the Euclidean space $\mathbb{R}^{n+1}$, and it has been proved that for any smooth convex initial hypersurface, the solution of the flow \eqref{eqH} converges smoothly to a round sphere. The analogue of the flow \eqref{eqH} in the hyperbolic space $\mathbb{H}^{n+1}$ was studied by Cabezas-Rivas and Miquel \cite{Cab-Miq2007} in 2007 assuming a stronger condition that the initial hypersurface is $h$-convex. There are further generalizations of the flow \eqref{eqH} in $\mathbb{H}^{n+1}$ with the mean curvature $H$ replaced by more general curvature functions including powers of the $k$th mean curvature $\sigma_k(\kappa),k=1,\cdots,n$ (see \cite{BenWei,Be-Pip2016,GLW-CAG,Mak2012,WX} for instance). In all cases, the $h$-convexity is assumed on the initial hypersurfaces.  This is mainly because that the $h$-convexity is convenient for the analysis of the curvature evolution equations and so that the tensor maximum principle or the constant rank theorem can be applied to derive that the $h$-convexity is preserved along the flow. Moreover, the $h$-convexity is strong enough geometrically such that the outer radius of the enclosed domain is uniformly controlled by its inner radius (see \cite{BM99}), and then the a prior $C^0$ estimate of the flow can be proved easily. In a recent work \cite{BenChenWei}, the first author with Andrews and Chen proved the smooth convergence of volume preserving $k$th mean curvature flows in $\mathbb{H}^{n+1}$ for initial hypersurfaces with positive sectional curvatures. This condition is weaker than $h$-convexity but still stronger than the convexity $\kappa_i>0.$

An open question in the field is whether the convexity ($\kappa_i>0$, $i=1,\cdots,n$) is sufficient to guarantee the smooth convergence of the volume preserving curvature flow in the hyperbolic space. Our result in Theorem \ref{theo} provides the first affirmative answer to this question. The curve case ($n=1$) of Theorem \ref{theo} was also treated recently by  the first and second authors in \cite{WY2022}, where the idea is that in this case there is only one curvature and we can calculate the curvature evolution explicitly to derive the convexity preserving immediately. However, the higher dimensional case ($n\geq 2$) as stated in Theorem \ref{theo} requires on some new ideas. The key step is again to show the convexity preserving, but the tensor maximum principle no longer works when applying to the evolution of the second fundamental form directly. Instead, we shall use the projection method via the Klein model of the hyperbolic space that was described earlier by the first author and Andrews in \cite{BenWei}. Based on this, we can treat the flow \eqref{flow-VMCF} as an equivalent flow in the Euclidean space and this allows us to derive a time-dependent positive lower bound on the principal curvatures, and a time-dependent upper bound on the Gauss curvature $K$. A continuation argument then implies that the flow exists for all positive time $t>0$.

To study the asymptotical behavior of the flow as time $t\to\infty$, we shall use some machinery from the theory of convex bodies in $\mathbb{H}^{n+1}$. In particular, we use the Blaschke selection theorem to show that for a subsequence of times $t_i\to\infty$, the enclosed domain $\Omega_{t_i}$ of $M_{t_i}$ converges in Hausdorff sense  to a limit convex domain $\hat{\Omega}$. We then apply the monotonicity of a certain quermassintegral to show that $\hat{\Omega}$ satisfies
\begin{equation}\label{s1-cur}
	\Phi_0(\hat{\Omega},\beta)=c\Phi_n(\hat{\Omega},\beta)
\end{equation}
for any Borel set $\beta$ in $\mathbb{H}^{n+1}$, where $\Phi_0$ and $\Phi_n$ are the hyperbolic curvature measures of $\hat{\Omega}$, and $c$ is a constant. A theorem of Kohlmann \cite{Koh1998} which characterises the geodesic ball in the hyperbolic space $\mathbb{H}^{n+1} (n\geq 2)$ using the equation \eqref{s1-cur} for curvature measures can be used to conclude that $\hat{\Omega}$ is a geodesic ball. Moreover, if we denote the center of the inner ball of $\Omega_t$ as $p_t$, then using the Alexandrov reflection argument and the subsequential Hausdorff convergence, we can deduce that for all time $t\to\infty$, the points $p_t$ converge to a fixed point $p$, which is the center of the ball $\hat{\Omega}$. With the help of this, we can prove the uniform positive bounds for the principal curvatures and then obtain the smooth convergence of the flow to a geodesic sphere.

The convergence result in Theorem \ref{theo} has an application in the Alexandrov-Fenchel inequalities for quermassintegrals in the hyperbolic space (see \S \ref{sec2} for the definitions). Along the flow \eqref{flow-VMCF}, we find that the $(n-1)$th quermassintegral $\mathcal{A}_{n-1}(\Omega_t)$ is monotone decreasing in time $t$. As the volume $|\Omega_t|$ is preserved, the smooth convergence proved in Theorem \ref{theo} yields the following Alexandrov-Fenchel inequality of quermassintegrals for convex domains in the hyperbolic space:
\begin{cor}\label{coro}
	Suppose that $\Omega\subset\mathbb{H}^{n+1}$ is a bounded weakly convex domain with smooth boundary $\partial\Omega$. Then the following inequality holds:
	\begin{equation}\label{eqA-F}
		\mathcal{A}_{n-1}(\Omega)\geq \psi_n\left(|\Omega|\right),
	\end{equation}
where $\psi_n:[0,\infty)\to \mathbb{R}$ is a strictly increasing function such that the equality is achieved for geodesic balls. Moreover, the equality holds in \eqref{eqA-F} if and only if $\Omega$ is a geodesic ball.
\end{cor}

Here we say that a hypersurface is weakly convex if all principal curvatures $\kappa_i\geq 0$. We remark that the inequality \eqref{eqA-F} for convex domains in the hyperbolic space was proved also by Brendle, Guan and Li in a preprint \cite{BGL} using a different method. The argument we present here provides a new proof of \eqref{eqA-F}.

The paper is organized as follows: In $\S$\ref{sec2}, we collect some preliminaries including the geometry of hypersufaces in the hyperbolic space, the evolution equations along the flow \eqref{flow-VMCF}, the quermassintegrals and curvature measures in the hyperbolic space. In \S \ref{subsec}, we project the flow to the Euclidean space via the Klein model. This projection has the advantage that the second fundamental forms of the corresponding solutions have a simple relation. In order to show that convexity is preserving along the flow \eqref{flow-VMCF}, it suffices to show that the corresponding flow \eqref{projected flow eq} in the Euclidean space preserves the convexity. In $\S$\ref{sec3}, we give the a priori $C^0$ and $C^1$ estimates of the flow \eqref{flow-VMCF} and the projected flow \eqref{projected flow eq}, which follows by a similar argument as in previous work \cite{BenWei,BenChenWei}. In $\S$\ref{sec4}, we prove the positive lower bound for the principal curvatures along the flow \eqref{flow-VMCF}, by proving the corresponding estimate along the flow \eqref{projected flow eq} in the Euclidean space. In \S \ref{sec.upK}, we adapt Tso's\cite{Tso85} technique to prove an upper bound on the Gauss curvature on any finite time interval. Since we only assumed convexity, the terms involving global term $\phi(t)$ need to be carefully treated. This implies two-sided curvature bounds of the solution on any finite time interval, and then we obtain the long time existence of the flow \eqref{flow-VMCF} in $\S$\ref{sec5}. In \S \ref{sec.hau}, we show the subsequential Hausdorff convergence of $M_t$ and the convergence of the center of the inner ball of $\Omega_t$ to a fixed point. Finally, in \S \ref{final}, we complete the proofs of Theorem \ref{theo} and Corollary \ref{coro}.
\begin{ack}
The research was surpported by National Key R and D Program of China 2021YFA1001800 and 2020YFA0713100, National Natural Science Foundation of China NSFC11721101 and Research grant KY0010000052 from University of Science and Technology of China. B. Yang would like to thank Professor Jiayu Li for his constant support.
\end{ack}

\section{Preliminaries}\label{sec2}
In this section, we collect some preliminary results concerning the geometry of hypersurfaces in the hyperbolic space, the evolution equations for geometric quantities along the flow \eqref{flow-VMCF}, the quermassintegrals and curvature measures in the hyperbolic space.
\subsection{Hyperbolic space}
The hyperbolic space $\mathbb{H}^{n+1}$ can be viewed as a warped product manifold $(\mathbb{R}_{+}\times \mathbb{S}^n,g_{\mathbb{H}^{n+1}})$ with
\begin{equation*}
	g_{\mathbb{H}^{n+1}}=d\rho^2+\sinh^2\rho g_{\mathbb{S}^n},
\end{equation*}
where $g_{\mathbb{S}^n}$ is the round metric on unit sphere $\mathbb{S}^n$. Let $D$ be the Levi-Civita connection on $\mathbb{H}^{n+1}$. The vector field $V=\sinh \rho\partial_\rho$ is a conformal Killing field satisfying $DV=\cosh \rho g_{\mathbb{H}^{n+1}}.$


Let $\Omega$ be a convex domain in $\mathbb{H}^{n+1}$ with a smooth boundary $M=\partial\Omega$. Then $M$ is a smooth convex hypersurface in $\mathbb{H}^{n+1}$. We denote by $g_{ij}, h_{ij}$ and $\nu$ the induced metric, the second fundamental form and the unit outward normal vector of $M$ respectively. As $M$ is convex, there exists a point $p_0\in\Omega$ such that $M$ is star-shaped with respect to $p_0$ and can be written as a radial graph $M=\{(\rho(\theta),\theta),~\theta\in \mathbb{S}^n\}$ with respect to $p_0$ for a smooth function $\rho\in C^\infty(\mathbb{S}^n)$. Equivalently, the support function of $M$ with respect to $p_0\in \Omega$ defined by
\begin{equation*}
  u=\langle V,\nu\rangle=\langle \sinh \rho\partial_\rho,\nu\rangle
\end{equation*}
is positive everywhere on $M$.  It is well known that (see e.g.\cite{GL15})
\begin{align*}
	g_{ij}&=\rho_{i}\rho_{j}+\sinh^2\rho\sigma_{ij},\\
	h_{ij}&=\frac{1}{\sqrt{\sinh^2\rho+|\bar{\nabla} \rho|^2}}(-(\sinh \rho)\rho_{ij}+2(\cosh \rho)\rho_i \rho_j+\sinh^2\rho\cosh \rho\sigma_{ij}),\\
\nu=&\frac{1}{\sqrt{1+|\bar{\nabla}\rho|^2/\sinh^2\rho}}\left(1,-\frac{\rho_1}{\sinh^2\rho},\cdots,-\frac{\rho_n}{\sinh^2\rho}\right),
\end{align*}
where $\bar{\nabla}$ is the covariant derivative on $\mathbb{S}^n$ with respect to the round metric $g_{\mathbb{S}^n}=(\sigma_{ij})$ and $\rho_i=\bar{\nabla}_i\rho, \rho_{ij}=\bar{\nabla}_i\bar{\nabla}_j\rho$. It follows that the Gauss curvature $K$ of $M$ can be expressed as a function of $\rho$ and its first and second derivatives:
\begin{equation}\label{eq-Gauss}
	K=\frac{\det h_{ij}}{\det g_{ij}}=\frac{\det(-\sinh \rho\rho_{ij}+2\cosh \rho\rho_i \rho_j+\sinh^2\rho\cosh \rho\sigma_{ij})}{(\sinh^2\rho+|\bar{\nabla} \rho|^2)^{\frac{n+2}{2}}(\sinh \rho)^{2n-2}}.
\end{equation}

\subsection{Evolution equations}
Let $M_t$ be a smooth solution to the curvature flow \eqref{flow-VMCF} in the hyperbolic space $\mathbb{H}^{n+1}$. We have the following evolution equations (see \cite{BenWei}) for the induced metric $g_{ij}$, the area element $d{\mu_t}$ and the speed function $K^{\alpha}$:
\begin{align}
	\frac{\partial}{\partial t}g_{ij}&=2\left(\phi(t)-K^{\alpha}\right)h_{ij},\label{eq-g}\\
	\frac{\partial}{\partial t}d\mu_t&=H\left(\phi(t)-K^{\alpha}\right)d\mu_t,\label{eq-dmu}\\
	\frac{\partial}{\partial t}K^{\alpha}&=\alpha K^{\alpha-1}\dot{K}^{ij}\left(\nabla_i\nabla_j{K^{\alpha}}+(K^{\alpha}-\phi(t))(h_i^k h_{kj}-\delta_{ij})\right),\label{eq-KK}
\end{align}
where $\dot{K}^{ij}$ denote the derivatives of $K$ with respect to the components of the second fundamental form, and $\nabla$ denotes the Levi-Civita connection on $M_t$ with respect to the induced metric $g_{ij}$. On the time interval when $M_t$ is star-shaped with respect to some point $p_0$, the support function $u(x,t)=\langle{\sinh \rho_{p_0}(x)\partial_{\rho_{p_0}},\nu}\rangle$ of $M_t$ with respect to $p_0$  evolves by (see \cite[Lemma 4.3]{BenWei}):
\begin{equation}\label{equeq}
	\frac{\partial}{\partial t}{u}=\alpha K^{\alpha-1}\dot{K}^{ij}\nabla_i\nabla_j{u}+\cosh{\rho_{p_0}}(x)\left(\phi(t)-(n\alpha+1)K^{\alpha}\right)+\alpha K^{\alpha}Hu.
\end{equation}

\subsection{Quermassintegrals}
Let $\mathcal{K}(\mathbb{H}^{n+1})$ be the set of compact convex sets in $\mathbb{H}^{n+1}$ with nonempty interior.  For any $\Omega\in \mathcal{K}(\mathbb{H}^{n+1})$ , the quermassintegrals of $\Omega$ are defined as follows (see \cite[Definition 2.1]{Sol05} \footnote{Note that the definition for $\mathcal{A}_k$ given here is the same as that for $W_{k+1}$ given in \cite{Sol05} up to a constant. In fact, we have $\mathcal{A}_k=(n+1)\binom{n}{k}W_{k+1}$.}):
\begin{equation}\label{Wk}
	\mathcal{A}_k(\Omega)=(n-k)\binom{n}{k}\frac{\omega_k\cdots\omega_0}{\omega_{n-1}\cdots\omega_{n-k-1}}\int_{\mathcal{L}_{k+1}}{\chi(L_{k+1}\cap\Omega)dL_{k+1}}
\end{equation}
for $k=0,1,\dots,n-1$, where $\omega_k=|\mathbb{S}^k|$ denotes the area of $k$-dimensional unit sphere,  $\mathcal{L}_{k+1}$ is the space of $(k+1)$-dimensional totally geodesic subspaces $L_{k+1}$ in $\mathbb{H}^{n+1}$ and $\binom{n}{k}=\frac{n!}{k!(n-k)!}$. The function $\chi$ is defined to be 1 if $L_{k+1}\cap\Omega\neq\emptyset$ and to be 0 otherwise. In particular, we have
\begin{equation*}
\mathcal{A}_{-1}(\Omega)=|\Omega|,\qquad \mathcal{A}_0(\Omega)=|\partial\Omega|,\qquad \mathcal{A}_{n}(\Omega)=\frac{\omega_{n}}{n+1}.
\end{equation*}

If the boundary $M=\partial\Omega$ is smooth (or at least of class $C^2$), we can define the principal curvatures $\kappa=(\kappa_1,\dots,\kappa_n)$ as the eigenvalues of the Weingarten matrix $\mathcal{W}$ of $M$. For each $k\in\{1,\dots,n\}$, the $k$th mean curvature $\sigma_k$ of $M$ is then defined as the $k$th elementary symmetric functions of the principal curvatures of $M$:
\begin{equation*}
	\sigma_k=\sum_{1\leq i_1<\cdots<i_k\leq n}{\kappa_{i_1}\cdots\kappa_{i_k}}.
\end{equation*}
These include the mean curvature $H=\sigma_1$ and Gauss curvature $K=\sigma_n$ as special cases.
In the smooth case, the quermassintegrals and the curvature integrals of a smooth convex domain $\Omega$ in $\mathbb{H}^{n+1}$ are related as follows:
\begin{align}%
	\mathcal{A}_1(\Omega)=&\int_{\partial\Omega}{\sigma_1}d\mu-n\mathcal{A}_{-1}(\Omega),\label{eq-V1}\\
	\mathcal{A}_k(\Omega)=&\int_{\partial\Omega}{\sigma_k}d\mu-\frac{n-k+1}{k-1}\mathcal{A}_{k-2}(\Omega),\quad k=2,\cdots,n \label{eq-VW}.
\end{align}
The quermassintegrals for smooth domains satisfy a nice variational property (see \cite{BA97}):
\begin{equation}\label{eqWk}
  \frac{d}{dt}\mathcal{A}_k(\Omega_t)=(k+1)\int_{M_t}\eta \sigma_{k+1}d\mu_t,\quad k=0,\cdots,n-1
\end{equation}
along any normal variation with velocity $\eta$.

The quermassintegrals defined by \eqref{Wk} are monotone with respect to inclusion of convex sets. That is, if $E,F\in \mathcal{K}(\mathbb{H}^{n+1})$ satisfy $E\subset F$, we have
\begin{equation}\label{s2.Akmo}
  \mathcal{A}_k(E)\leq \mathcal{A}_k(F)
\end{equation}
for all $k=0,1,\cdots,n$. Moreover, they are continuous with respect to the Hausdoff distance. Recall that the Hausdorff distance between two convex sets $\Omega, L\in \mathcal{K}(\mathbb{H}^{n+1})$ is defined as
\begin{equation*}
	\mathrm{dist}_{\mathcal{H}}(\Omega,L):=\mathrm{inf}\{\lambda>0:\Omega\subset B_{\lambda}(L)\,  \text{and}\, L\subset B_{\lambda}(\Omega)\},
\end{equation*}
where $	B_{\lambda}(L):=\{q\in \mathbb{H}^{n+1}|~\mathrm{d}_{\mathbb{H}^{n+1}}(q,L)<\lambda\}$.
\begin{lem}\label{inquer}
	Suppose that $\{\Omega_i\}_{i=1}^\infty, \Omega\in\mathcal{K}(\mathbb{H}^{n+1})$ and $\Omega_i$ converges to $\Omega$ in the Hausdorff sense. Then we have $\lim_{i\to\infty}{\mathcal{A}_k(\Omega_i)}=\mathcal{A}_k(\Omega)$ for all $k=-1,0,\cdots, n-1$.
\end{lem}

To see this, we use an expression of the measure $dL_{k+1}$. Every totally geodesic subspace $L_{k+1}$ in $\mathbb{H}^{n+1}$ is determined by its orthogonal subspace $L_{n-k}[o]$, which passes through the origin $o\in\mathbb{H}^{n+1}$, and by the intersection point $m=L_{k+1}\cap L_{n-k}[o]$. In this way, $\mathcal{L}_{k+1}$ can be identified as a bundle over the Grassmannian manifolds $G_{n+1,n-k}$ which consists of all subspaces $L_{n-k}[o]$, and then $dL_{k+1}$ is written as (see \cite{San04})
\begin{equation}
	dL_{k+1}=\cosh^{k+1}(\rho) d\mu_{n-k} dL_{n-k}[o],
\end{equation}
where $d\mu_{n-k}$ is the volume element on $L_{n-k}[o]$, $dL_{n-k}[o]$ is the volume element on $G_{n+1,n-k}$ and $\rho=d_{\mathbb{H}^{n+1}}(x,o)$ denotes the distance of a point $x$ in $L_{n-k}[o]$ to the origin. Then we can rewrite the integral in \eqref{Wk} equivalently as
\begin{equation}\label{s2.Ak2}
	\int_{\mathcal{L}_{k+1}}{\chi(L_{k+1}\cap\Omega)dL_{k+1}}=\int_{G_{n+1,n-k}}\left(\int_{\Pi_{L_{n-k}[o]}(\Omega)}\cosh^{k+1}(\rho) d\mu_{n-k}\right)dL_{n-k}[o],
\end{equation}
where $\Pi_{L_{n-k}[o]}:\mathbb{H}^{n+1}\to\mathbb{H}^{n+1}$ is the orthogonal projection over $L_{n-k}[o]\in G_{n+1,n-k}$ along geodesics. Then the conclusion in Lemma \ref{inquer} follows from \eqref{s2.Ak2} immediately.

\subsection{Curvature measures in $\mathbb{H}^{n+1}$}For any $\Omega
\in \mathcal{K}(\mathbb{H}^{n+1})$ and $\varepsilon>0$, define the parallel set
\begin{equation*}
	\Omega_{\varepsilon}=\{x\in\mathbb{H}^{n+1}|~d_{\mathbb{H}^{n+1}}(x,\Omega)\leq\varepsilon\}.
\end{equation*}
The maps $f_\Omega:\mathbb{H}^{n+1}\setminus \Omega\to\partial \Omega$ and $F_\Omega:\mathbb{H}^{n+1}\setminus \Omega\to T_{\partial \Omega}\mathbb{H}^{n+1}$:
\begin{equation*}
	d_{\mathbb{H}^{n+1}}(f_\Omega(x),x)=d_{\mathbb{H}^{n+1}}(x,\Omega) \quad \text{and} \quad x=\mathrm{exp}_{f_\Omega(x)}^{\mathbb{H}^{n+1}}(d(\Omega,x)F_{\Omega}(x))
\end{equation*}
are well-defined. For any Borel set $\beta\subset\mathbb{H}^{n+1}$, define the local parallel set
\begin{equation*}
	M_{\varepsilon}(\Omega,\beta)=f_{\Omega}^{-1}(\beta\cap \partial \Omega)\cap(\Omega_{\varepsilon}\setminus \Omega).
\end{equation*}
\begin{figure}
  \centering
 \begin{tikzpicture}
\begin{scope}[even odd rule]
\clip (-3,-3.4) [rounded corners=10pt] -- (-3.6,-2.6) -- (-3.5,-1.4)
 [rounded corners=12pt]--(-2.8,-0.5)
 [rounded corners=12pt]--(-1.7,-0.6)
 [rounded corners=12pt]--(-0.3,-1.2)
 [rounded corners=9pt]--(-0.1,-1.9)
  [rounded corners=12pt]--(-0.4,-2.7)--(-1.9,-3.6)
  [rounded corners=10pt]--cycle (-3,-3.9) [rounded corners=10pt] -- (-3.9,-2.9) --(-4.1,-1.7)-- (-3.6,-0.8)
 [rounded corners=12pt]--(-2.8,-0.1)
 [rounded corners=12pt]--(-1.7,0)
 [rounded corners=12pt]--(-0.2,-0.5)
 [rounded corners=9pt]--(0.4,-1.9)
  [rounded corners=12pt]--(0.2,-2.7)--(-1,-3.8)
  [rounded corners=10pt]--cycle;
\path[fill=green!20,dashed,draw] (-1.75,0.2) -- (-1.9,-1) -- (-0.6,-1.2) -- (-0.3,-0.3) --cycle;
\end{scope}
\path[fill=green!5,draw]
(-3,-3.4) [rounded corners=10pt] -- (-3.6,-2.6) -- (-3.5,-1.4)
 [rounded corners=12pt]--(-2.8,-0.5)
 [rounded corners=12pt]--(-1.7,-0.6)
 [rounded corners=12pt]--(-0.3,-1.2)
 [rounded corners=9pt]--(-0.1,-1.9)
  [rounded corners=12pt]--(-0.4,-2.7)--(-1.9,-3.6)
  [rounded corners=10pt]--cycle;
\path[draw,dashed,thick,color=blue]
(-3,-3.9) [rounded corners=10pt] -- (-3.9,-2.9) --(-4.1,-1.7)-- (-3.6,-0.8)
 [rounded corners=12pt]--(-2.8,-0.1)
 [rounded corners=12pt]--(-1.7,0)
 [rounded corners=12pt]--(-0.2,-0.5)
 [rounded corners=9pt]--(0.4,-1.9)
  [rounded corners=12pt]--(0.2,-2.7)--(-1,-3.8)
  [rounded corners=10pt]--cycle;

\node at (-1.2,-2.9) {$\Omega$};
\node at (0.8,-1.9) {$\Omega_\varepsilon$};
\draw[-stealth,dashed] (-1,-0.6) to [out=80,in=170] (0.5,0.4) node[right]{$M_{\varepsilon}(\Omega,\beta)$};
\draw[-stealth,dashed] (-1.5,-0.7) to [out=-170,in=80] (-2,-1.4) node[below]{$\beta\cap\partial \Omega$};
\end{tikzpicture}
 \caption{Local parallel set $M_{\varepsilon}(\Omega,\beta)$}
\end{figure}
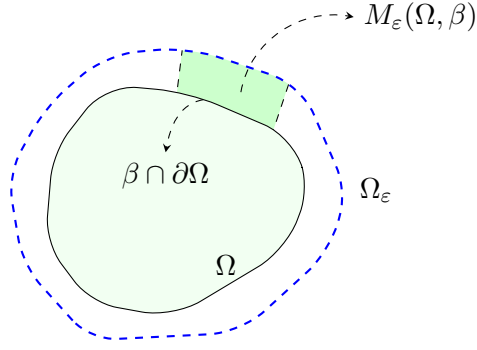
Following \cite{Koh1991}, given a convex set $\Omega$ and $\varepsilon>0$, let us define a Radon measure $\mu_\varepsilon$ on the Borel $\sigma$-algebra $\mathcal{B}(\mathbb{H}^{n+1})$ of the hyperbolic space by
\begin{equation*}
	\mu_{\varepsilon}(\Omega,\beta)=\mathrm{Vol}_{\mathbb{H}^{n+1}}(M_{\varepsilon}(\Omega,\beta)).
\end{equation*}
Set
\begin{equation*}
	l_{k+1}(t)=\int_0^{t}{\sinh^{k}(s)\cosh^{n-k}(s) ds},\quad k=0,1,\cdots,n
\end{equation*}
and $l_0\equiv 1$. Then there holds the following Steiner-type formula:
\begin{equation*}
	\mu_{\varepsilon}(\Omega,\beta)=\sum_{k=0}^{n}l_{n+1-k}(\varepsilon)\Phi_{k}(\Omega,\beta),\quad \forall\beta\in\mathcal{B}(\mathbb{H}^{n+1}).
\end{equation*}
The coefficients $\Phi_0(\Omega,\cdot),\cdots,\Phi_n(\Omega,\cdot)$ are called the curvature measures of the convex body $\Omega$, which are Borel measures on $\mathbb{H}^{n+1}$.

When the boundary $\partial \Omega$ is smooth, $\Phi_k(\Omega,\cdot)$ has a nice expression:
\begin{equation*}
	\Phi_k(\Omega,\beta)=\int_{\partial \Omega\cap\beta}\sigma_{n-k}d\mu,
\end{equation*}
where  $\sigma_{n-k}$ is the $(n-k)$th mean curvature of $\partial \Omega$.  We refer \cite{Koh1991} to the readers for a general integral representation for $\Phi_k(\Omega,\cdot)$.

As in the classical setting of convex bodies in the Euclidean space, the curvature measures introduced by Kohlmann are solid enough to be weakly continuous with respect to the topology induced on $\mathcal{K}(\mathbb{H}^{n+1})$.
\begin{thm}[\cite{VG-2019}]\label{s2.thmcurv}
	Let $\{\Omega_j\}_{j=1}^{\infty}\subset \mathcal{K}(\mathbb{H}^{n+1})$ be a sequence of convex sets such that $\Omega_j\to \Omega$ as $j\to\infty$ in the Hausdorff topology. Then for every $k=0,\cdots,n$ we have
	\begin{equation*}
		\Phi_k(\Omega_j,\cdot)\to\Phi_k(\Omega,\cdot)
	\end{equation*}
	as $j\to\infty$, weakly in the sense of measure.
\end{thm}

The following characterization of geodesic balls via the hyperbolic curvature measures was proved by Kohlmann \cite{Koh1998} and can be viewed as a generalization of the classical Alexandrov Theorem in differential geometry.

\begin{thm}[\cite{Koh1998}]\label{gAT}
Let $n\geq 2$ and $k\in\{0,\dots,n-2\}$. Assume that $\Omega$ is a compact, connected, locally convex subset of $\mathbb{H}^{n+1}$ with nonempty interior satisfying
\begin{equation}\label{Ath}
	\Phi_{k}(\Omega,\beta)=c\Phi_n(\Omega,\beta)
\end{equation}
for any Borel set $\beta\in\mathcal{B}(\mathbb{H}^{n+1})$, where $c>0$ is a constant, then $\Omega$ is a geodesic ball.
\end{thm}
Note that there is no smoothness assumption on the domain $\Omega$ in Theorem \ref{gAT}. When the boundary $\partial\Omega$ is smooth, the equation \eqref{Ath} means that the $(n-k)$th mean curvature $\sigma_{n-k}$ is a constant on $\partial\Omega$, and the conclusion of Theorem \ref{gAT} in this case reduces to the classical result of Montiel and Ros \cite{MS91}.

\section{Klein model and Projection}\label{subsec}
To investigate the flow \eqref{flow-VMCF} for convex hypersurfaces, it is convenient to  project it to the Euclidean space $\mathbb{R}^{n+1}$ and apply the Gauss map parametrization as in the work \cite[\S 5]{BenWei} by the first author and Andrews. We briefly review the argument here and refer the readers to \cite{BenWei} for details.

Let us denote by $\mathbb{R}^{1,n+1}$ the Minkowski spacetime, that is the vector space $\mathbb{R}^{n+2}$ endowed with the Minkowski spacetime metric $\langle\cdot,\cdot\rangle$ given by
\begin{equation*}
\langle X,X\rangle~=~-X_0^2+\sum_{i=1}^{n+1}X_i^2
\end{equation*}
for any vector $X=(X_0,X_1,\cdots,X_{n+1})\in\mathbb{R}^{n+2}$. The hyperbolic space $\mathbb{H}^{n+1}$ is then
\begin{equation*}
\mathbb{H}^{n+1}=\{X\in\mathbb{R}^{1,n+1}~|~\langle X,X\rangle=-1,\ X_0>0\}.
\end{equation*}

The Klein model parametrizes the hyperbolic space using the unit disc, which induces a projection from an embedding $X:M^n\rightarrow\mathbb{H}^{n+1}$ to an embedding $Y:M^n\rightarrow B_1(0)\subset\mathbb{R}^{n+1}$ by
\begin{equation}\label{s3.proj}
X~=~\frac{(1,Y)}{\sqrt{1-|Y|^2}}.
\end{equation}
Let $\nu\in T\mathbb{H}^{n+1}, g^X_{ij}, h^X_{ij}, K^X$ and $N\in\mathbb{R}^{n+1}, g^Y_{ij}, h^Y_{ij}, K^Y$ denote the unit normal vectors, induced metrics, the second fundamental forms and Gauss curvatures of $X(M^n)\subset\mathbb{H}^{n+1}$ and $Y(M^n)\subset \mathbb{R}^{n+1}$ respectively. We have the following relations:
\begin{align}
h^X_{ij}&=\frac{h^Y_{ij}}{\sqrt{\left(1-|Y|^2\right)\left(1-\langle N,Y\rangle^2\right)}},\label{eq-hij}\\
g^X_{ij}&=\frac{1}{1-|Y|^2}\left(g^Y_{ij}+\frac{\langle Y,\partial_i Y\rangle\langle Y,\partial_j Y\rangle}{1-|Y|^2}\right).\label{eq-metric}
\end{align}
It follows from \eqref{eq-hij} that $X(M^n)$ is convex in $\mathbb{H}^{n+1}$ is equivalent to that $Y(M^n)$ is convex in $\mathbb{R}^{n+1}$. If we take the determinant on both sides of \eqref{eq-hij} and \eqref{eq-metric}, the Gauss curvatures satisfy
\begin{equation}\label{s3.K}
  K^X=\left(\frac{1-|Y|^2}{1-\langle N,Y\rangle^2}\right)^{\frac{n+2}{2}}K^Y.
\end{equation}

Suppose that $X: M^n\times [0,T)\to \mathbb{H}^{n+1}$ is a smooth convex solution to the flow \eqref{flow-VMCF}. Then the corresponding solution $Y:M^n\times [0,T)\to B_1(0)\subset\mathbb{R}^{n+1}$ satisfies the evolution equation:
\begin{equation}\label{s3.evlY}
  \frac{\partial}{\partial t}Y=\sqrt{(1-|Y|^2)(1-\langle N,Y\rangle^2)}\Big(\phi(t)-(K^X)^{\alpha}\Big)N.
\end{equation}
Since each $Y_t(M^n)=Y(M^n,t)\subset B_1(0)\subset \mathbb{R}^{n+1}$ is convex, we can parametrize $Y_t(M^n)$ using the Gauss map and the support function $s(z):=\langle Y_t(N^{-1}(z)),z\rangle$, where $N^{-1}:\mathbb{S}^n\to M^n$ is the inverse map of the Gauss map. Then $Y_t(M^n)$ is given by the embedding $Y:\mathbb{S}^n\to \mathbb{R}^{n+1}$ with (see \cite[\S 2]{Ur90})
\begin{equation}\label{s3.Y-Gas}
  Y(z)=~s(z)z+\bar{\nabla} s
\end{equation}
where $\bar{\nabla}$ is the gradient with respect to the round metric $g_{\mathbb{S}^n}$  on $\mathbb{S}^n$. The derivative of this map is given by
\begin{equation*}
  \partial_iY=~\mathfrak{r}_{ik}\sigma^{kl}\partial_lz
\end{equation*}
in local coordinates, where $ \mathfrak{r}_{ij}$ is given as follows
\begin{equation}\label{s2.r-def}
  \mathfrak{r}_{ij}=~\bar{\nabla}_i\bar{\nabla}_js+s\sigma_{ij}.
\end{equation}
The eigenvalues $ \mathfrak{r}_i$ of $ \mathfrak{r}_{ij}$ with respect to the round metric $g_{\mathbb{S}^n}$ are the inverse of the principal curvatures $\kappa_i^Y$, i.e., $ \mathfrak{r}_i=1/{\kappa_i^Y}$, and are called the principal radii of curvature. The determinant of \eqref{s2.r-def} gives the Gauss curvature of $Y_t(M^n)$:
\begin{equation}\label{s3.Ky}
  K^Y=\frac 1{\mathrm{det}(\mathfrak{r}_{ij})}.
\end{equation}
By the identity \eqref{s3.K}, we obtain
\begin{equation}\label{eq-K}
K^X=\left(\frac{1-(s^2+|\bar{\nabla} s|^2)}{1-s^2}\right)^\frac{n+2}{2}K^Y.
\end{equation}
The evolution equation \eqref{s3.evlY} is equivalent to  the following scalar parabolic equation:
\begin{equation*}
  \frac{\partial}{\partial t}s=~\sqrt{(1-s^2-|\bar{\nabla}s|^2)(1-s^2)}\Big(\phi(t)-(K^X)^{\alpha}\Big)
\end{equation*}
on the round sphere $\mathbb{S}^n$ for the support function $s(z,t)$ of $Y(M^n,t)$.

We summarize the results as the following lemma:
\begin{lem}
Assume that $X(M^n,t)\subset \mathbb{H}^{n+1}, t\in [0,T)$ is a smooth convex solution to the flow \eqref{flow-VMCF}. Then up to a tangential diffeomorphism,  the corresponding solution $Y(M^n,t)\subset B_1(0)\subset \mathbb{R}^{n+1}$ and its support function $s(z,t)$ satisfy the following equations:
	\begin{align}
		\frac{\partial}{\partial t}Y=&~\sqrt{(1-|Y|^2)(1-\langle N,Y\rangle^2)}\Big(\phi(t)-(K^X)^{\alpha}\Big)N,\label{projected flow eq}\\
		\frac{\partial}{\partial t}s=&~\sqrt{(1-s^2-|\bar{\nabla}s|^2)(1-s^2)}\Big(\phi(t)-(K^X)^{\alpha}\Big),\label{eq-sptf0}
	\end{align}
where the Gauss curvature $K^X$ is related to the Gauss curvature $K^Y$ via the identities \eqref{s3.K} and \eqref{eq-K}.
\end{lem}

For simplicity of the notation, we denote $A$ and $B$ as the following functions:
\begin{equation}\label{eqAB}
	\left\{\begin{aligned}
		A=A(s,\bar{\nabla} s)&:=\sqrt{(1-s^2-|\bar{\nabla} s|^2)(1-s^2)},\\
		B=B(s,\bar{\nabla} s)&:=(1-s^2-|\bar{\nabla} s|^2)^{\frac{n+2}{2}\alpha+\frac{1}{2}}(1-s^2)^{-\frac{n+2}{2}\alpha+\frac{1}{2}}.
	\end{aligned}\right.
\end{equation}
Then \eqref{eq-sptf0} can be simplified as:
\begin{equation}
 \frac{\partial}{\partial t}s=A\phi(t)-B(K^{Y})^{\alpha}.\label{eq-sptf}
\end{equation}

In \S \ref{sec4}, we will estimate the upper bound of the largest eigenvalues of $\mathfrak{r}_{ij}$ along the equation \eqref{eq-sptf} and this implies a positive lower bound on the principal curvatures of the solution $M_t$ to the original flow \eqref{flow-VMCF}. In the following lemma, we derive the evolution equation of $\mathfrak{r}_{ij}$ along the flow \eqref{eq-sptf}:
\begin{lem}\label{s3.lem2}
Let $s(z,t)$ be a smooth solution of the flow \eqref{eq-sptf}. Choose a local orthonormal frame $e_1,\cdots, e_n$ around the point $z\in \mathbb{S}^n$ such that $(\mathfrak{r}_{ij})$ is diagonal at $(z,t)$ and $\mathfrak{r}_{11}$ corresponds to the largest eigenvalue of $(\mathfrak{r}_{ij})$. Along the equation \eqref{eq-sptf}, we have the following evolution equation of $\mathfrak{r}_{11}$:
\begin{align}\label{s2.tau11}
	\frac{\partial}{\partial t}\mathfrak{r}_{11}-F^{k\ell}\bar{\nabla}_k\bar{\nabla}_\ell\mathfrak{r}_{11}=&\phi(t)\Big(A_{s_1s_1}{\mathfrak{r}_{11}}^2+2A_{ss_1}s_1\mathfrak{r}_{11}-2A_{s_1s_1}\mathfrak{r}_{11}s+A_s\mathfrak{r}_{11}+A_{s_k}\bar{\nabla}_k{\mathfrak{r}_{11}}\notag\\
	&+A+A_{ss}{s_1}^2-2A_{ss_1}ss_1-A_ss+A_{s_1s_1}s^2-A_{s_1}s_1\Big)\notag\\
&-(K^Y)^{\alpha}\Big(B_{s_1s_1}{\mathfrak{r}_{11}}^2+2B_{ss_1}s_1\mathfrak{r}_{11}-2B_{s_1s_1}\mathfrak{r}_{11}s+B_s\mathfrak{r}_{11}\notag\\
	&+B+B_{ss}{s_1}^2-2B_{ss_1}ss_1-B_ss+B_{s_1s_1}s^2-B_{s_1}s_1\notag\\
	&+B_{s_k}\bar{\nabla}_k{\mathfrak{r}_{11}}+\alpha B\mathfrak{r}^{kk}\mathfrak{r}^{\ell\ell}(\bar{\nabla}_1{\mathfrak{r}_{k\ell}})^2+{\alpha}^2 B(\sum_k\mathfrak{r}^{kk}\bar{\nabla}_1\mathfrak{r}_{kk})^2\notag\\
	&-2\alpha \mathfrak{r}^{kk}\bar{\nabla}_1\mathfrak{r}_{kk}(B_{s_1}\mathfrak{r}_{11}+B_ss_1-B_{s_1}s)+\alpha B\mathfrak{r}^{kk}(\mathfrak{r}_{11}-\mathfrak{r}_{kk})\Big),
\end{align}
where $F^{k\ell}=\alpha B(K^Y)^{\alpha}\mathfrak{r}^{k\ell}$.
The notations such as $A_s, A_{s_k}$ are partial derivatives of $A$ with respect to its first and second arguments.
\end{lem}
\proof
We first recall that $\mathfrak{r}_{ij}$ satisfies a Codazzi-type identity
\begin{equation}\label{s2.codz}
  \bar{\nabla}_k\mathfrak{r}_{ij}=\bar{\nabla}_i\mathfrak{r}_{kj}.
\end{equation}
This follows from the Ricci identity for commuting covariant derivatives on $\mathbb{S}^n$. Differentiating \eqref{s2.codz} and commuting derivatives, we have the Simons' identity for the second derivatives
\begin{equation}\label{s2.Sim}
  \bar{\nabla}_{(i}\bar{\nabla}_{j)}\mathfrak{r}_{k\ell}=\bar{\nabla}_{(k}\bar{\nabla}_{\ell)}\mathfrak{r}_{ij}+\mathfrak{r}_{k\ell}\sigma_{ij}-\mathfrak{r}_{ij}\sigma_{k\ell},
\end{equation}
where the brackets denote symmetrisation.  Let $(\mathfrak{r}^{ij})=(\mathfrak{r}_{ij})^{-1}$ be the inverse of $(\mathfrak{r}_{ij})$. The spatial derivatives of $K^Y$ satisfy
\begin{align}
  \bar{\nabla}_iK^Y= & -\frac{1}{(\mathrm{det} (\mathfrak{r}_{ij}))^2} \bar{\nabla}_i\mathrm{det} (\mathfrak{r}_{ij})\nonumber\\
  =& -\frac{1}{\mathrm{det} (\mathfrak{r}_{ij})}\mathfrak{r}^{k\ell}\bar{\nabla}_i\mathfrak{r}_{k\ell}\nonumber\\
  = & -K^Y \mathfrak{r}^{k\ell}\bar{\nabla}_i\mathfrak{r}_{k\ell},\label{s2.dK}
\end{align}
and
\begin{align}
\bar{\nabla}_j\bar{\nabla}_iK^Y=&-K^Y \mathfrak{r}^{k\ell}\bar{\nabla}_j\bar{\nabla}_i\mathfrak{r}_{k\ell}-\bar{\nabla}_jK^Y \mathfrak{r}^{k\ell}\bar{\nabla}_i\mathfrak{r}_{k\ell}-K^Y \bar{\nabla_j}\mathfrak{r}^{k\ell}\bar{\nabla}_i\mathfrak{r}_{k\ell}\nonumber\\
  =&-K^Y \mathfrak{r}^{k\ell}\bar{\nabla}_j\bar{\nabla}_i\mathfrak{r}_{k\ell}+K^Y\left(\mathfrak{r}^{k\ell}\bar{\nabla}_i\mathfrak{r}_{k\ell}\right)\left(\mathfrak{r}^{pq}\bar{\nabla}_j\mathfrak{r}_{pq}\right)\nonumber\\
  &\quad +K^Y\mathfrak{r}^{kq}\mathfrak{r}^{p\ell}\bar{\nabla}_j\mathfrak{r}_{pq}\bar{\nabla}_i\mathfrak{r}_{k\ell}. \label{s2.d2K}
\end{align}
On the other hand, the spatial derivatives of $A(s,\bar{\nabla}s)$  can be calculated as
\begin{align}
  \bar{\nabla}_iA= & A_ss_i+A_{s_k}s_{ki} \nonumber\\
  =&A_{s_k}\mathfrak{r}_{ki}+A_ss_i-sA_{s_k}\sigma_{ki},\label{s2.dA}
\end{align}
and
\begin{align}
   \bar{\nabla}_j\bar{\nabla}_iA =& A_ss_{ij}+A_{ss}s_is_j+A_{ss_\ell}s_{\ell j}s_i+A_{s_k s}s_{j}s_{ki}+A_{s_ks_\ell}s_{ki}s_{\ell j}+A_{s_k}s_{kij}\nonumber\\
   =& A_s(\mathfrak{r}_{ij}-s\sigma_{ij})+A_{ss}s_is_j+A_{ss_\ell}(\mathfrak{r}_{\ell j}-s\sigma_{\ell j})s_i\nonumber\\
   &+A_{s_k s}s_{j}(\mathfrak{r}_{ki}-s\sigma_{ki})+A_{s_ks_\ell}(\mathfrak{r}_{ki}-s\sigma_{ki})(\mathfrak{r}_{\ell j}-s\sigma_{\ell j})\nonumber\\
   &+A_{s_k}\left(\bar{\nabla}_j\mathfrak{r}_{ki}-s_j\sigma_{ki}\right)\nonumber\\
   =&A_{s_k}\bar{\nabla}_k\mathfrak{r}_{ij}+A_s\mathfrak{r}_{ij}+A_{ss_\ell}s_i\mathfrak{r}_{\ell j}+A_{s_k s}s_{j}\mathfrak{r}_{ki}+A_{s_ks_\ell}\mathfrak{r}_{ki}\mathfrak{r}_{\ell j}\nonumber\\
   &-A_{s_k s_l}\mathfrak{r}_{ki}\sigma_{\ell j}s-A_{s_k s_\ell}\mathfrak{r}_{\ell j}\sigma_{ki}s-A_ss\sigma_{ij}+A_{ss}s_is_j-A_{ss_{\ell}}\sigma_{\ell j}ss_i\nonumber\\
   &-A_{s_k s}\sigma_{ki}s_{j}s+A_{s_k s_\ell}\sigma_{ki}\sigma_{lj}s^2-A_{s_k}\sigma_{ki}s_j,\label{s2.d2A}
\end{align}
where $A_s, A_{s_k}$ are partial derivatives of $A$ with respect to its first and second arguments, and we used \eqref{s2.r-def} and the Codazzi-type identity \eqref{s2.codz}. The calculation for $B(s,\bar{\nabla}s)$ is similar.

Taking the spatial derivatives to the equation \eqref{eq-sptf} and using \eqref{s2.dK}$-$\eqref{s2.d2A}, we compute that
\begin{align}\label{s2.d2S}
	\bar{\nabla}_1\bar{\nabla}_1(\partial_ts)
=&\phi(t)\bar{\nabla}_1\bar{\nabla}_1A-(K^{Y})^{\alpha}\bar{\nabla}_1\bar{\nabla}_1B-2\alpha (K^Y)^{\alpha-1}\bar{\nabla}_1K^Y\bar{\nabla}_1B\nonumber\\
&-\alpha (K^Y)^{\alpha-1}\bar{\nabla}_1\bar{\nabla}_1K^Y B-\alpha(\alpha-1)(K^Y)^{\alpha-2}(\bar{\nabla}_1K^Y)^2B\notag\\
=&\phi(t)\bar{\nabla}_1\bar{\nabla}_1A-(K^{Y})^{\alpha}\Big(\bar{\nabla}_1\bar{\nabla}_1B-2\alpha \mathfrak{r}^{k\ell}\bar{\nabla}_1\mathfrak{r}_{k\ell}\bar{\nabla}_1B\nonumber\\
&- \alpha B \mathfrak{r}^{k\ell}\bar{\nabla}_1\bar{\nabla}_1\mathfrak{r}_{k\ell}+\alpha^2 B \big(\sum_{k,\ell}\mathfrak{r}^{k\ell}\bar{\nabla}_1\mathfrak{r}_{k\ell}\big)^2+\alpha B\mathfrak{r}^{kq}\mathfrak{r}^{p\ell}\bar{\nabla}_1\mathfrak{r}_{pq}\bar{\nabla}_1\mathfrak{r}_{k\ell}\Big)\nonumber\\
	=&\phi(t)\Big(A_{s_k}\bar{\nabla}_k{\mathfrak{r}_{11}}+A_s\mathfrak{r}_{11}+2A_{ss_1}s_1\mathfrak{r}_{11}+A_{s_1s_1}{\mathfrak{r}_{11}^2}-2A_{s_1s_1}\mathfrak{r}_{11}s\nonumber\\
&-A_ss+A_{ss}{s_1}^2-2A_{ss_1}ss_1+A_{s_1s_1}s^2-A_{s_1}s_1\Big)\notag\\ &-(K^Y)^{\alpha}\Big(B_{s_k}\bar{\nabla}_k{\mathfrak{r}_{11}}+B_s\mathfrak{r}_{11}+2B_{ss_1}s_1\mathfrak{r}_{11}+B_{s_1s_1}{\mathfrak{r}_{11}^2}-2B_{s_1s_1}\mathfrak{r}_{11}s\nonumber\\
&-B_ss+B_{ss}{s_1}^2-2B_{ss_1}ss_1+B_{s_1s_1}s^2-B_{s_1}s_1\notag\\
	&-2\alpha \mathfrak{r}^{k\ell}\bar{\nabla}_1\mathfrak{r}_{k\ell}(B_{s_1}\mathfrak{r}_{11}+B_ss_1-B_{s_1}s)-\alpha B\mathfrak{r}^{k\ell}\bar{\nabla}_1\bar{\nabla}_1{\mathfrak{r}_{k\ell}}\notag\\
	&+{\alpha}^2 B(\sum_{k,\ell}\mathfrak{r}^{k\ell}\bar{\nabla}_1\mathfrak{r}_{k\ell})^2+\alpha B\mathfrak{r}^{ik}\mathfrak{r}^{j\ell}\bar{\nabla}_1{\mathfrak{r}_{ij}}\bar{\nabla}_1{\mathfrak{r}_{k\ell}}\Big).
\end{align}

Let $F^{k\ell}=\alpha B(K^Y)^{\alpha}\mathfrak{r}^{k\ell}$. Taking derivatives of \eqref{s2.r-def},
\begin{align}\label{s3.1}
\frac{\partial}{\partial t}\mathfrak{r}_{11}-F^{k\ell}\bar{\nabla}_k\bar{\nabla}_\ell\mathfrak{r}_{11}=&\bar{\nabla}_1\bar{\nabla}_1(\partial_ts)+\partial_ts-F^{k\ell}\bar{\nabla}_k\bar{\nabla}_\ell\mathfrak{r}_{11}.
\end{align}
Substituting  \eqref{eq-sptf}, \eqref{s2.Sim}, \eqref{s2.d2S} into \eqref{s3.1}, using the fact that $(\mathfrak{r}_{k\ell})$ is diagonal at the point we are considering and rearranging the terms, we obtain the equation \eqref{s2.tau11}.
\endproof

We also calculated the evolution equation of $r^2=|Y|^2=s^2+|\bar{\nabla}s|^2$.
\begin{lem}\label{s3.lem3}
Let $s(z,t)$ be a smooth solution of the flow \eqref{eq-sptf}. If we choose a local orthonormal frame $e_1,\cdots, e_n$ around the point $z\in \mathbb{S}^n$, then we have:
\begin{align}\label{s4.dr2}
		&\frac{\partial}{\partial t}r^2-F^{k\ell}\bar{\nabla}_k\bar{\nabla}_\ell r^2\nonumber\\
		=&2\phi(t)\Big(sA+A_s|\bar{\nabla} s|^2+A_{s_k}s_i\mathfrak{r}_{ki}-A_{s_k}s_ks\Big)\notag\\
		-&2(K^Y)^{\alpha}\Big((1-n\alpha)sB+B_s|\bar{\nabla} s|^2+B_{s_k}s_i\mathfrak{r}_{ki}-B_{s_k}s_ks+\alpha B\sum_{k}{\mathfrak{r}_{kk}}\Big),
\end{align}
where $F^{k\ell}=\alpha B(K^Y)^{\alpha}\mathfrak{r}^{k\ell}$.
\end{lem}
\proof
Since $r^2=s^2+|\bar{\nabla}s|^2$, using \eqref{eq-sptf}, \eqref{s2.dK} and \eqref{s2.dA}, we have
\begin{align*}
	\partial_tr^2=&2ss_t+2s_is_{it}\notag\\
=&2s(A\phi(t)-B(K^Y)^{\alpha})+2s_i\Big(\phi(t)\bar{\nabla}_iA-(K^Y)^{\alpha}\bar{\nabla}_iB+\alpha B(K^Y)^{\alpha}\mathfrak{r}^{k\ell}\bar{\nabla}_i\mathfrak{r}_{k\ell}\Big)\nonumber\\
	=&2\phi(t)\left(sA+A_s|\bar{\nabla} s|^2+A_{s_k}s_i\mathfrak{r}_{ki}-A_{s_k}s_i\sigma_{ki}s\right)\nonumber\\
	&-2(K^Y)^{\alpha}\left(sB+B_s|\bar{\nabla} s|^2+B_{s_k}s_i\mathfrak{r}_{ki}-B_{s_k}s_i\sigma_{ki}s-\alpha B\mathfrak{r}^{k\ell}\bar{\nabla}_i\mathfrak{r}_{k\ell}s_i\right),
\end{align*}
and
\begin{align*}
  \bar{\nabla}_k\bar{\nabla}_\ell r^2=&2 \bar{\nabla}_k\left(ss_\ell+s_is_{i\ell}\right)\nonumber\\
  =&2s_ks_\ell+2ss_{k\ell}+2s_{ik}s_{i\ell}+2s_is_{ik\ell}\nonumber\\
  =&2\mathfrak{r}_{ik}\mathfrak{r}_{i\ell}-2s\mathfrak{r}_{k\ell}+2s_i\bar{\nabla}_i\mathfrak{r}_{k\ell}.
\end{align*}
Combining the above two equations gives the equation \eqref{s4.dr2}.
\endproof

We conclude the section with following remarks.
\begin{rem} (i). The projection method via Klein model has also been used recently by Chen and Huang \cite{ChenHuang} to study the contracting Gauss curvature flows in the hyperbolic space.

(ii). Denote $\mathcal{W}^X$ the Weingarten matrix of $X(M^n)\subset \mathbb{H}^{n+1}$. Then \eqref{eq-hij} and \eqref{eq-metric} imply that the inverse matrix $\mathcal{W}_X^{-1}$ of $\mathcal{W}^X$ satisfies
\begin{align}\label{s5:W-inv}
  (\mathcal{W}_X^{-1})_{ij} =& (h^{-1}_X)^{jk}g_{ki}^X \nonumber\\
   =&(h^{-1}_Y)^{kj}\left(g_{ki}^Y+\frac{\langle  Y,\partial_iY\rangle\langle Y,\partial_kY\rangle }{(1-|Y|^2)}\right)\sqrt{\frac{1-\langle N,Y\rangle^2}{1-|Y|^2}}.
\end{align}
We can also express \eqref{s5:W-inv} using $s,\bar{\nabla}s$ and $\mathfrak{r}_{ij}$ as follows:
\begin{align}\label{s5:W-inv-2}
  (\mathcal{W}_X^{-1})_{ij}     =&\left(\sigma^{jq}+\frac{\langle \sigma^{ja}\bar{\nabla}_as,\sigma^{qb}\bar{\nabla}_bs\rangle}{1-s^2-|\bar{\nabla}s|^2}\right) \mathfrak{r}_{qi}\sqrt{\frac{1-s^2}{1-s^2-|\bar{\nabla}s|^2}}.
\end{align}
Since the eigenvalues of $(\mathcal{W}_X^{-1})_{ij}$ correspond to the reciprocal of the principal curvatures $\kappa_1,\cdots,\kappa_n$ of $X(M^n,t)\subset \mathbb{H}^{n+1}$, the above equation means that in order to estimate the lower bound of $\kappa_i$, it suffices to estimate the upper bound of $\mathfrak{r}_{ij}$ together with the $C^0, C^1$ estimates of $s$. This observation will be used in the next section. Moreover, we believe that \eqref{s5:W-inv-2} would also be useful for studying general fully nonlinear curvature flows in the hyperbolic space via the projection method described in this section.
\end{rem}

\section{A priori estimate}
In this section, we first review the $C^0$ and $C^1$ estimates of the solution $M_t$ to the flow \eqref{flow-VMCF}, which can be proved using a similar argument as in \cite{BenChenWei,BenWei}.  Then we derive a positive lower bound on the principal curvatures $\kappa_i$ of the solution $M_t$, and prove the upper bound of the Gauss curvature $K$. Both of them may depend on the time $t$. 

\subsection{$C^0$ and $C^1$ estimates}\label{sec3}
Let $M_0$ be a smooth, closed and convex hypersurface in the hyperbolic space. The flow \eqref{flow-VMCF} has a unique smooth solution $M_t$ for at least a short time and $M_t$ is convex on a possibly shorter time interval.  Without loss of generality, we suppose that $M_t$ a smooth convex solution to the flow \eqref{flow-VMCF} starting from $M_0$ on a maximum time interval $[0,T)$, where $T\leq \infty$.

Denote by $\Omega_t$ as the domain enclosed by $M_t$ such that $\partial\Omega_t=M_t$. As the velocity of the flow \eqref{flow-VMCF} only depends on the curvature which is invariant under reflection with respect to a totally geodesic hyperplane, we can argue as in \cite[\S 4]{BenChenWei} using the Alexandrov reflection method to show that the inner radius and outer radius of $\Omega_t$ are uniformly bounded.
\begin{lem}\label{in-out}
	Let $M_t$ be the smooth convex solution to the flow \eqref{flow-VMCF} on the time interval $t\in [0,T)$. Denote $\rho_-(t)$, $\rho_+(t)$ be the inner radius and outer radius of $\Omega_t$. Then there exist positive constants $c_1,\ c_2$ depending only on $n$, $\alpha$ and $M_0$ such that
	\begin{equation}\label{in-out-radius}
	0<c_1\leq\rho_-(t)\leq\rho_+(t)\leq c_2
	\end{equation}
	for all time $t\in[0,T)$.
\end{lem}

By \eqref{in-out-radius}, the inner radius of $\Omega_t$ is bounded below by a positive constant $c_1$. This implies that for each $t\in[0,T)$  there exists a geodesic ball $B_{c_1}(p_t)$ of radius $c_1$ centered at some point $p_t$ such that $B_{c_1}(p_t)\subset\Omega_t$. A similar argument as in \cite[Lemma 4.2]{BenWei} yields the existence of a geodesic ball with fixed center enclosed by the flow hypersurface on a suitable fixed time interval.
\begin{lem}\label{lem3.2}
	Let $M_t$ be the smooth convex solution to the flow \eqref{flow-VMCF} on the time interval $[0,T)$. For any $t_0\in[0,T)$, let $B_{\rho_0}(p_0)$ be the inball of $\Omega_{t_0}$, where $\rho_0=\rho_-(t_0)$. Then
	\begin{equation}\label{inball-t0}
	B_{\rho_0/2}(p_0)\subset\Omega_t,\ \ t\in[t_0,\min\{T,t_0+\tau\})
	\end{equation}
	for some $\tau$ depending only on $n$, $\alpha$ and $M_0$.
\end{lem}

Since the projection \eqref{s3.proj} is a diffeomorphism, by the equation \eqref{eq-hij} we know that if the flow \eqref{flow-VMCF} has a smooth convex solution on the maximal time interval $[0,T)$, then so does the projected flow \eqref{projected flow eq}.  Given $t_0\in[0,T)$ and let $B_{\rho_0}(p_0)$ be the inball of $\Omega_{t_0}$, where $\rho_0=\rho_-(t_0)$.

Consider the support function $u(x,t)=\sinh \rho_{p_0}\langle\partial_{\rho_{p_0}},\nu\rangle$ of $M_t$ with respect to the point $p_0$, where $\rho_{p_0}$ is the distance function in $\mathbb{H}^{n+1}$ from the point $p_0$. Since $M_t$ is convex, by \eqref{in-out-radius} and \eqref{inball-t0}, we see
\begin{equation}\label{u-bound}
	\begin{split}
u(x,t)&\geq\sinh\left(\frac{\rho_0}{2}\right)\geq\sinh\left(\frac{c_1}{2}\right)=:2c,\\
u(x,t)&\leq\sinh(2c_2),
    \end{split}
\end{equation}
and
\begin{equation}
0<\frac{c_1}{2}\leq\rho_{p_0}(t)\leq 2c_2<\infty
\end{equation}
for any $t\in[t_0,\min\{T,t_0+\tau\})$. Assume that $p_0$ is the origin of $\mathbb{H}^{n+1}$, we project the flow \eqref{flow-VMCF} in $\mathbb{H}^{n+1}$ onto $B_1(0)\subset \mathbb{R}^{n+1}$ with respect to the original point $p_0$. Then on the time interval $[t_0,\min\{T,t_0+\tau\})$, the Euclidean distance $r$ from the origin satisfies $r=\tanh{\rho_{p_0}}$ and is uniformly bounded from below and above. Under the Gauss map parametrization \eqref{s3.Y-Gas}, the support function $s$ of $Y_t(M^n)$ satisfies
\begin{equation*}
   r^2=s^2+|\bar{\nabla} s|^2.
\end{equation*}
We conclude with the following $C^0$ and $C^1$ estimates:
\begin{lem}\label{t0-bound on u,du}
	Let $M_t$ be the smooth convex solution of the flow \eqref{flow-VMCF} on the time interval $[0,T)$. Given $t_0\in[0,T)$ and $\tau$ be defined as in Lemma \ref{lem3.2}. Let $B_{\rho_0}(p_0)$ be the inball of $\Omega_{t_0}$ with $\rho_0=\rho_-(t_0)$ and we project $\mathbb{H}^{n+1}$ onto $B_1(0)\subset \mathbb{R}^{n+1}$ with respect to the original point $p_0$. We have the following estimates of the projected flow \eqref{projected flow eq}:
	\begin{equation}
	0<c_3\leq 1-r^2\leq 1-s^2\leq c_4<1,\qquad  |s|\leq c_5<1,\qquad |\bar{\nabla} s|\leq c_6<1
	\end{equation}
	for any $t\in[t_0,\min\{T,t_0+\tau\})$, where $\{c_i\}_{i=3,4,5,6}$ are positive constants depending only on $n$, $\alpha$ and $M_0$.
\end{lem}


\subsection{Preserving convexity}\label{sec4}
In this subsection, we prove the lower bound on the principal curvatures of the solution $M_t$ of the flow \eqref{flow-VMCF}.
\begin{prop}\label{preserve convex}
Let $M_0$ be a smooth, closed and convex hypersurface in $\mathbb{H}^{n+1}$ and $M_t, t\in [0,T)$, be the smooth solution of the flow \eqref{flow-VMCF} starting from $M_0$. Then there exist constants $\Lambda_1$ and $\Lambda_3$ depending only on $n$, $\alpha$ and $M_0$ such that the principal curvatures $\kappa_i$ of $M_t$ satisfy
\begin{equation}\label{s4.2-0}
  \kappa_i\geq \Lambda_3^{-1}\Lambda_1^{-\frac{2t}{\tau}}
\end{equation}
 for all $i\in\{1,\dots,n\}$ and $t\in[0,T)$, where $\tau$ is the constant in Lemma \ref{lem3.2}.
\end{prop}
\proof
Since $M_0$ is convex, by continuity the solution $M_t$ is convex for at least a short time. Without loss of generality, we assume that the solution is convex on the time interval $[0,T)$ and aim to derive the estimate \eqref{s4.2-0} on this interval $[0,T)$. In fact, if $T_1<T$ is the largest time such that $M_t$ is convex for all $0\leq t<T_1$, then the estimate \eqref{s4.2-0} implies that the principal curvatures of $M_{T_1}$ has a positive lower bound and this contradicts with the maximality of $T_1$. Therefore, $M_t$ remains convex for all time $t\in [0,T)$ and satisfies the estimate \eqref{s4.2-0}.

So in the following, we will derive the estimate \eqref{s4.2-0} on the interval $[0,T)$ where the solution $M_t$ is convex.  Given any fixed time $t_0\in[0,T)$ and $\tau$ be defined as in Lemma \ref{lem3.2}.  Let $B_{\rho_0}(p_0)$ be the inball of $\Omega_{t_0}$, where $\rho_0=\rho_-(t_0)$. We consider the solution $Y_t(M^n)$ of the projected flow \eqref{projected flow eq} on $\mathbb{R}^{n+1}$ with respect to the original point $p_0$. Let $\mathfrak{r}_{ij}=s_{ij}+s\sigma_{ij}$ be the inverse of the Weingarten matrix of the projected hypersurface $Y_t$ with respect to an orthonormal frame on $\mathbb{S}^n$, whose eigenvalues $(\mathfrak{r}_1,\dots,\mathfrak{r}_n)$ are the principal radii of curvature of $Y_t$. As the principal curvatures $\kappa_i$ of $M_t=X(M^n,t)$ are the eigenvalues of the Weingarten matrix $\mathcal{W}_X$, by the expression \eqref{s5:W-inv-2} and the $C^0, C^1$ estimates in Lemma \ref{t0-bound on u,du}, in order to prove the lower bound of $\kappa_i$, it suffices to prove the upper bound of $\mathfrak{r}(z,t)=\max_{i=1,\dots,n}{\mathfrak{r}_i(z,t)}$.

For this purpose, we consider the function
\begin{equation*}
 {G}(z,t):=\log{\mathfrak{r}}+\frac{L}{2}r^2,\qquad t\in [t_0,\min\{T,t_0+\tau\}),
\end{equation*}
 where $L>0$ is a large constant to be determined. Suppose that the maximum of ${G}$ on $\mathbb{S}^n\times [t_0,\min\{T,t_0+\tau\})$ is attained at $(\bar{z},\bar{t})$. We choose a local orthonormal frame $e_1,\dots,e_n$ around $\bar{z}$ such $\{\mathfrak{r}_{ij}(z,t)\}$ is diagonal at $(\bar{z},\bar{t})$ and $\mathfrak{r}_{11}(\bar{z},\bar{t})=\mathfrak{r}_1(\bar{z},\bar{t})=\mathfrak{r}(\bar{z},\bar{t})$. Then without loss of generality, we can view the function $G$ as the following form
\begin{equation}\label{eqG}
	G(z,t)=\log{\mathfrak{r}_{11}(z,t)}+\frac{L}{2}r^2.	
\end{equation}
If $\bar{t}=t_0$, we have
 \begin{equation}\label{s4.G1}
   G(z,t)\leq \max_{z\in\mathbb{S}^n}G(z,t_0)~\leq \log\max_{z\in\mathbb{S}^n}\mathfrak{r}_{11}(z,t_0)+\frac{L}{2}(1-c_3)
 \end{equation}
for $(z,t)\in \mathbb{S}^n\times [t_0,\min\{T,t_0+\tau\})$. In the following, we assume $\bar{t}>t_0$.  We shall apply the maximum principle to the evolution equation of $G$ to derive the upper bound of $\mathfrak{r}_{11}$.

We have calculated the evolution equations of $\mathfrak{r}_{11}$ and of $r^2$ in Lemma \ref{s3.lem2} and Lemma \ref{s3.lem3}. Since $(\bar{z},\bar{t})$ is a maximum point of $G$, at $(\bar{z},\bar{t})$ there hold
\begin{align}\label{s4.dQ}
	0=\bar{\nabla}_iG=&\frac{\bar{\nabla}_i\mathfrak{r}_{11}}{\mathfrak{r}_{11}}+L(ss_i+s_ks_{ki}) =\frac{\bar{\nabla}_i\mathfrak{r}_{11}}{\mathfrak{r}_{11}}+Ls_k\mathfrak{r}_{ki}
\end{align}
and
\begin{align}\label{s4.dtQ}
		0&\leq \partial_tG-F^{k\ell}\bar{\nabla}_k\bar{\nabla}_\ell G\nonumber\\
		&=\frac{1}{\mathfrak{r}_{11}}(\partial_t \mathfrak{r}_{11}-F^{k\ell}\bar{\nabla}_k\bar{\nabla}_\ell \mathfrak{r}_{11})+\frac{L}{2}(\partial_tr^2-F^{k\ell}\bar{\nabla}_k\bar{\nabla}_\ell r^2)+F^{k\ell}\frac{\bar{\nabla}_k\mathfrak{r}_{11}\bar{\nabla}_\ell \mathfrak{r}_{11}}{(\mathfrak{r}_{11})^2}\nonumber\\
		&=:Q_1+Q_2,
\end{align}
where $Q_1$ involves the term $\phi(t)$:
\begin{align}\label{s4.Q1-0} Q_1=&\frac{\phi(t)}{\mathfrak{r}_{11}}\Big(A_{s_1s_1}{\mathfrak{r}_{11}}^2+2A_{ss_1}s_1\mathfrak{r}_{11}-2A_{s_1s_1}\mathfrak{r}_{11}s+A_s\mathfrak{r}_{11}+A_{s_k}\bar{\nabla}_k{\mathfrak{r}_{11}}\notag\\
	&+A+A_{ss}{s_1}^2-2A_{ss_1}ss_1-A_ss+A_{s_1s_1}s^2-A_{s_1}s_1\Big)\notag\\
&+\phi(t)L\Big(sA+A_s|\bar{\nabla} s|^2+A_{s_k}s_i\mathfrak{r}_{ki}-A_{s_k}s_i\sigma_{ki}s\Big),
\end{align}
and $Q_2$ involves the term $(K^Y)^{\alpha}$:
\begin{align}\label{s4.Q2-0}
	Q_2=&-\frac{(K^Y)^{\alpha}}{{\mathfrak{r}}_{11}}\Big(B_{s_1s_1}{\mathfrak{r}_{11}}^2+2B_{ss_1}s_1\mathfrak{r}_{11}-2B_{s_1s_1}\mathfrak{r}_{11}s+B_s\mathfrak{r}_{11}\notag\\
	&+B+B_{ss}{s_1}^2-2B_{ss_1}ss_1-B_ss+B_{s_1s_1}s^2-B_{s_1}s_1\notag\\
	&+B_{s_k}\bar{\nabla}_k{\mathfrak{r}_{11}}+\alpha B\mathfrak{r}^{kk}\mathfrak{r}^{\ell\ell}(\bar{\nabla}_1{\mathfrak{r}_{k\ell}})^2+{\alpha}^2 B(\sum_k\mathfrak{r}^{kk}\bar{\nabla}_1\mathfrak{r}_{kk})^2\notag\\
	&-2\alpha \mathfrak{r}^{kk}\bar{\nabla}_1\mathfrak{r}_{kk}(B_{s_1}\mathfrak{r}_{11}+B_ss_1-B_{s_1}s)+\alpha B\mathfrak{r}^{kk}(\mathfrak{r}_{11}-\mathfrak{r}_{kk})\Big)\nonumber\\
&-(K^Y)^{\alpha}L\Big((1-n\alpha)sB+B_s|\bar{\nabla} s|^2+B_{s_k}s_i\mathfrak{r}_{ki}-B_{s_i}s_is+\alpha B\sum_{k}{\mathfrak{r}_{kk}}\Big)\nonumber\\
&+F^{k\ell}\frac{\bar{\nabla}_k\mathfrak{r}_{11}\bar{\nabla}_\ell \mathfrak{r}_{11}}{(\mathfrak{r}_{11})^2}.
\end{align}

In the following, we estimate the terms $Q_1$ and $Q_2$ separately. First, we use \eqref{s4.dQ} to cancel the term on $\bar{\nabla}_k{\mathfrak{r}_{11}}$ in \eqref{s4.Q1-0} and get
\begin{align}\label{s4.Q1}
Q_1=&\frac{\phi(t)}{\mathfrak{r}_{11}}\Big(A_{s_1s_1}{\mathfrak{r}_{11}}^2+2A_{ss_1}s_1\mathfrak{r}_{11}-2A_{s_1s_1}\mathfrak{r}_{11}s+A_s\mathfrak{r}_{11}\notag\\
	&+A+A_{ss}{s_1}^2-2A_{ss_1}ss_1-A_ss+A_{s_1s_1}s^2-A_{s_1}s_1\Big)\notag\\
&+\phi(t)L\Big(sA+A_s|\bar{\nabla} s|^2-A_{s_i}s_is\Big).
\end{align}
Since
\begin{equation*}
	A_{s_1s_1}=-\frac{(1-s^2)^{\frac{1}{2}}}{(1-{s}^2-|\bar{\nabla}{s}|^2)^{\frac{3}{2}}}\left(1-({s}^2+|\bar{\nabla}{s}|^2)+s_1^2\right)<0
\end{equation*}
and the coefficients such as $A, A_s, A_{ss_1}, \cdots$ depend only on $s$ and $\bar{\nabla}s$ which are uniformly bounded on the interval $[t_0, \min\{T,t_0+\tau\})$, we obtain the estimate:
\begin{align}\label{s4.Q1'}
Q_1	\leq&~ \frac{\phi(t)}{{\mathfrak{r}}_{11}}\left(a_1+(a_2+a_3L){\mathfrak{r}}_{11}-a_4 {\mathfrak{r}}_{11}^2\right),
\end{align}
where $a_i>0,i=1,2,3,4$ depend only on $n,\alpha$ and $M_0$.

To estimate $Q_2$, we use \eqref{s4.dQ} to kill the terms involving $\bar{\nabla}_k\mathfrak{r}_{11}$ and get
\begin{align*}
&  -\frac{(K^Y)^{\alpha}}{{\mathfrak{r}}_{11}} \alpha B\mathfrak{r}^{kk}\mathfrak{r}^{\ell\ell}(\bar{\nabla}_1{\mathfrak{r}_{k\ell}})^2 +F^{k\ell}\frac{\bar{\nabla}_k\mathfrak{r}_{11}\bar{\nabla}_\ell \mathfrak{r}_{11}}{(\mathfrak{r}_{11})^2} \\
   \leq & -\frac{(K^Y)^{\alpha}}{{\mathfrak{r}}_{11}} \alpha B\mathfrak{r}^{kk}\mathfrak{r}^{11}(\bar{\nabla}_k{\mathfrak{r}_{11}})^2 +\alpha B (K^Y)^\alpha \mathfrak{r}^{kk}\frac{(\bar{\nabla}_k\mathfrak{r}_{11})^2}{(\mathfrak{r}_{11})^2}=0.
\end{align*}
Then
\begin{align}\label{s4.Q2}
	Q_2\leq &-\frac{(K^Y)^{\alpha}}{{\mathfrak{r}}_{11}}\Big(B_{s_1s_1}{\mathfrak{r}_{11}}^2+\left(2B_{ss_1}s_1-2B_{s_1s_1}s+B_s\right)\mathfrak{r}_{11}\notag\\
	&+B+B_{ss}{s_1}^2-2B_{ss_1}ss_1-B_ss+B_{s_1s_1}s^2-B_{s_1}s_1\notag\\
	&+{\alpha}^2 B(\sum_k\mathfrak{r}^{kk}\bar{\nabla}_1\mathfrak{r}_{kk})^2-2\alpha \mathfrak{r}^{kk}\bar{\nabla}_1\mathfrak{r}_{kk}(B_{s_1}\mathfrak{r}_{11}+B_ss_1-B_{s_1}s)\Big)\nonumber\\
&-(K^Y)^{\alpha}L\Big((1-n\alpha)sB+B_s|\bar{\nabla} s|^2-B_{s_i}s_is+\alpha B\mathfrak{r}_{11}\Big).
\end{align}
The third line on the right hand side of \eqref{s4.Q2} can be estimated using Cauchy-Schwarz inequality:
\begin{align*}
 &- {\alpha}^2 B(\sum_k\mathfrak{r}^{kk}\bar{\nabla}_1\mathfrak{r}_{kk})^2+2\alpha \mathfrak{r}^{kk}\bar{\nabla}_1\mathfrak{r}_{kk}(B_{s_1}\mathfrak{r}_{11}+B_ss_1-B_{s_1}s)  \\
  \leq  & \frac{1}{B}\left(B_{s_1}\mathfrak{r}_{11}+B_ss_1-B_{s_1}s\right)^2.
\end{align*}
Since the coefficients such as $B, B_{s_i}, B_s, B_{ss_i}$ are bounded and $B>0$, we obtain the estimate
\begin{align}\label{s4.Q2'}
	Q_2	\leq&\frac{(K^Y)^{\alpha}}{{\mathfrak{r}}_{11}}\bigg(b_1+(b_2+b_3L){\mathfrak{r}}_{11}+(b_4-b_5L){\mathfrak{r}}_{11}^2\bigg),
\end{align}
where $b_i>0,i=1,2,3,4,5$ depend only on $n,\alpha$ and $M_0$.

Combining the two estimates \eqref{s4.Q1'} and \eqref{s4.Q2'} and choosing $L={2b_4}/{b_5}$, we can get an uniform upper bound ${\mathfrak{r}}_{11}(\bar{z},\bar{t})\leq \Lambda$, where $\Lambda$ depends only on $n$, $\alpha$ and $M_0$. It follows from
\begin{align*}
	G(z,t)&\leq \max\left\{G(\bar{z},\bar{t}),\max_{z\in\mathbb{S}^n}G(z,t_0)\right\}
\end{align*}
that
\begin{align}\label{tau1}
  \mathfrak{r}_{11}(z,t)&\leq \max\left\{\mathfrak{r}_{11}(\bar z,\bar t),\max_{z\in\mathbb{S}^n}\mathfrak{r}_{11}(z,t_0)\right\}\exp\Big(\frac{L}{2}\big(1-c_3-r^2(z,t)\big)\Big)\notag\\
  &\leq \max\left\{\Lambda,\max_{z\in\mathbb{S}^n}\mathfrak{r}_{11}(z,t_0)\right\}\exp\Big(\frac{L}{2}(c_4-c_3)\Big)\notag\\
  &=\Lambda_1\max\left\{\Lambda,\max_{z\in\mathbb{S}^n}\mathfrak{r}_{11}(z,t_0)\right\}
\end{align}
for all $(z,t)\in \mathbb{S}^n\times [t_0, \min\{T,t_0+\tau\})$, where $\Lambda_1=\exp\big(\frac{L}{2}(c_4-c_3)\big)\geq 1$ and $c_3,c_4$ are constants in Lemma \ref{t0-bound on u,du}.

Note that $t_0\in [t_0-\frac{\tau}{2},\min\{T,t_0+\frac{\tau}{2}\})$. Applying the above argument for the time interval $[t_0-\frac{\tau}{2},\min\{T,t_0+\frac{\tau}{2}\})$ gives
\begin{equation}\label{tau2}
 \max_{z\in\mathbb{S}^n}\mathfrak{r}_{11}(z,t_0)\leq \Lambda_1\max\left\{\Lambda,\max_{z\in\mathbb{S}^n}\mathfrak{r}_{11}(z,t_0-\frac{\tau}{2})\right\}.
\end{equation}
Combining \eqref{tau1}, \eqref{tau2} and the fact that $\Lambda_1\geq 1$, we have
\begin{equation*}
	\mathfrak{r}_{11}(z,t)\leq \Lambda_1^2\max\left\{\Lambda,\max_{z\in\mathbb{S}^n}\mathfrak{r}_{11}(z,t_0-\frac{\tau}{2})\right\}.
\end{equation*}
By backward induction on the time interval, we finally get
\begin{align}\label{tauup}
	\mathfrak{r}_{11}(z,t)&\leq \Lambda_1^{\left[\frac{2t_0}{\tau}\right]+2}\max\{\Lambda,\max_{z\in\mathbb{S}^n}\mathfrak{r}_{11}(z,0)\}\notag\\
	&\leq \Lambda_1^{\frac{2t}{\tau}+2}\max\{\Lambda,\max_{z\in\mathbb{S}^n}\mathfrak{r}_{11}(z,0)\}~=:\Lambda_2\Lambda_1^{\frac{2t}{\tau}}
\end{align}
for all $(z,t)\in \mathbb{S}^n\times [t_0, \min\{T,t_0+\tau\})$, where $[\cdot]$ denotes the integer part of a real constant,  and $\Lambda_1$, $\Lambda_2=\Lambda_1^{2}\max\{\Lambda,\max_{z\in\mathbb{S}^n}\mathfrak{r}_{11}(z,0)\}$ are constants depending only on $n$, $\alpha$ and $M_0$. Since $t_0$ is arbitrary, by \eqref{s5:W-inv-2} and the $C^0,C^1$ estimates, we conclude that the principal curvatures $\kappa_i$ of the solution $M_t$ of the flow \eqref{flow-VMCF} satisfy
\begin{equation*}
	\kappa_i\geq \Lambda_3^{-1}\Lambda_1^{-\frac{2t}{\tau}},\qquad i=1,\dots,n,
\end{equation*}
for all time $t\in [0,T)$, where $\Lambda_3=C\Lambda_2$ for a constant $C$ which estimates the bound on the coefficient of \eqref{s5:W-inv-2} involving $s$ and $\bar{\nabla}s$.
\endproof

\subsection{Upper bound of Gauss curvature} \label{sec.upK}
Now we use the technique of Tso \cite{Tso85} to prove the upper bound of the Gauss curvature of the solution $M_t$ along the flow \eqref{flow-VMCF}.
\begin{prop}\label{propKupp}
	Let $M_t,\ t\in[0,T)$ be the smooth solution of the flow \eqref{flow-VMCF} starting from a smooth closed convex hypersurface $M_0$. If $T<\infty$, then there is a constant $C$ depending on $n, \alpha, M_0$ and $T$ such that the Gauss curvature $K$ of $M_t$  satisfies
\begin{equation*}
	\max_{M_t} K\leq C
\end{equation*}
for any $t\in [0,T)$.
\end{prop}
\proof
For any given $t_0\in[0,T)$, let $B_{\rho_0}(p_0)$ be the inball of $\Omega_{t_0}$, where $\rho_0=\rho_{-}(t_0)$. Consider the support function $u(x,t)=\sinh \rho_{p_0}(x)\langle{\partial_{\rho_{p_0}},\nu}\rangle$ of $M_t$ with respect to point $p_0$, where $\rho_{p_0}(x)$ is the distance function in $\mathbb{H}^{n+1}$ from the point $p_0$. Since $M_t$ is convex for all $t\in [0,T)$, by \eqref{u-bound} we have
\begin{equation}\label{equbound}
	2c\leq u\leq \sinh(2c_2)
\end{equation}
on $M_t$ for any $t\in \left[t_0,\min\{T,t_0+\tau\}\right)$. We define the auxiliary function
\begin{equation*}
	W=\frac{K^\alpha}{u-c},
\end{equation*}
which is well-defined for time $t\in \left[t_0,\min\{T,t_0+\tau\}\right)$. We shall apply the maximum principle to the evolution equation of $W$ to derive the upper bound of $K$.

Combining \eqref{eq-KK} and \eqref{equeq}, we compute that along the flow \eqref{flow-VMCF}  $W$ evolves as
\begin{align}
		\frac{\partial}{\partial t}{W}
		=&\alpha K^{\alpha-1}\dot{K}^{ij}(W_{ij}+\frac{2}{u-c}u_iW_j)\nonumber\\
&\quad -\frac{\phi(t)}{u-c}\bigg(\alpha K^{\alpha-1}(HK-\sigma_{n-1}(\kappa))+W\cosh\rho_{{p_0}}(x)\bigg)\notag\\
		&\quad +(1+n\alpha)\frac{K^{2\alpha}}{(u-c)^2}\cosh\rho_{{p_0}}(x)-\frac{c\alpha K^{2\alpha}}{(u-c)^2}H-\alpha WK^{\alpha-1}\sigma_{n-1}(\kappa)\notag\\
		\leq&~\alpha K^{\alpha-1}\dot{K}^{ij}(W_{ij}+\frac{2}{u-c}u_iW_j)+\underbrace{\frac{\phi(t)}{u-c}\alpha K^{\alpha-1}\left(\sigma_{n-1}(\kappa)-HK\right)}_{(I)}\notag\\
		&\qquad +(1+n\alpha)W^2\cosh\rho_{{p_0}}(x)-\alpha cHW^2.\label{eqW}
\end{align}
In the previous work \cite{BenWei,BenChenWei}, $M_t$ is assumed to be either $h$-convex ($\kappa_i>1, ~i=1,\cdots,n$) or positively curved ($\kappa_i\kappa_j>1$,~$\forall~i\neq j$), so the terms $(I)$ involving $\phi(t)$ can be thrown away when we estimate the upper bound of $W$. But in our case we only have convexity, so we still need to estimate the terms $(I)$ in \eqref{eqW} carefully.

Let $\tilde{W}(t)=\max_{M_t}W(x,t)$. Noting that $K^\alpha=(u-c)W$, by the definition \eqref{eqphi} of $\phi(t)$ and the upper bound \eqref{equbound} of $u$, we have
\begin{equation*}
	\phi(t)=\frac{1}{|M_t|}\int_{M_t}K^{\alpha}\,\mathrm{d}t\leq \max_{M_t}{K^{\alpha}(\cdot,t)}\leq (\sinh(2c_2)-c)\tilde{W}.
\end{equation*}
By the lower bound on the principal curvatures in Lemma \ref{preserve convex}, we also have
\begin{align*}
	\sigma_{n-1}(\kappa)&=K(\frac{1}{\kappa_1}+\cdots\frac{1}{\kappa_n})\leq nK(\min_{1\leq i\leq n} \kappa_i)^{-1}\leq nK\Lambda_3\Lambda_1^{{2T}/{\tau}}.
\end{align*}
It follows that the terms $(I)$ can be estimated as
\begin{equation}\label{s4.2-1}
	(I)\leq ~n\alpha (\sinh(2c_2)-c)\Lambda_3\Lambda_1^{{2T}/{\tau}}\tilde{W}^2~=:~\beta_1(T)\tilde{W}^2,
\end{equation}
where for simplicity of the notations we denote $\beta_1(T)=n\alpha (\sinh(2c_2)-c)\Lambda_3\Lambda_1^{{2T}/{\tau}}$, which depends on the maximal existence time $T<\infty$.

The first term on the last line of \eqref{eqW} can be simply estimated from above by
\begin{equation}\label{s4.2-2}
  (1+n\alpha)W^2\cosh\rho_{{p_0}}(x)~\leq (1+n\alpha)\cosh(2c_2)\tilde{W}^2.
\end{equation}
The last term of \eqref{eqW} provides sufficient negative term, since by $H\geq nK^{1/n}$ there holds:
\begin{align}\label{s4.2-3}
  -\alpha cHW^2\leq  & -\alpha nc K^{1/n}W^2\nonumber\\
  = & -\alpha nc(u-c)^{\frac{1}{n\alpha}}W^{2+\frac{1}{n\alpha}}\nonumber\\
  \leq &-\alpha nc^{1+\frac{1}{n\alpha}}W^{2+\frac{1}{n\alpha}}.
\end{align}
Combining \eqref{s4.2-1} - \eqref{s4.2-3}, we arrive at
\begin{align}\label{s4.2-4}
  \frac{d}{dt}\tilde{W} \leq & \bigg(\beta_1(T)+(1+n\alpha)\cosh(2c_2)-\alpha nc^{1+\frac{1}{n\alpha}}\tilde{W}^{\frac{1}{n\alpha}}\biggr)\tilde{W}^2.
\end{align}

The coefficient of the hightest order term on the right hand side of \eqref{s4.2-4} is negative, so the comparison principle implies that $\tilde{W}$ is bounded above by a positive constant. In fact, whenever
\begin{equation*}
\tilde{W}\geq \bigg(\frac{2}{\alpha n}\left(\beta_1(T)+(1+n\alpha)\cosh(2c_2)\right)\bigg)^{n\alpha}c^{-n\alpha-1}~=:\beta_2(T),
\end{equation*}
we have
\begin{align*}
  \frac{d}{dt}\tilde{W} \leq & -\frac{\alpha n}{2}c^{1+\frac{1}{n\alpha}}\tilde{W}^{2+\frac{1}{n\alpha}}.
\end{align*}
Therefore,
\begin{equation}\label{s4.2-5}
  \tilde{W}(t)\leq \max\left\{\bigg(W^{-1-\frac{1}{n\alpha}}(t_0)+\frac{n\alpha+1}{2}c^{1+\frac{1}{n\alpha}}(t-t_0)\bigg)^{-\frac{n\alpha}{n\alpha+1}},\beta_2(T)\right\}
\end{equation}
for all time $t\in \left[t_0,\min\{T,t_0+\tau\}\right)$.

For $t_0=0$, we obtain from \eqref{s4.2-5} the upper bound
\begin{equation*}
	\tilde{W}(t)\leq\max\left\{\tilde{W}(0),\beta_2(T)\right\},\quad \forall t\in[0,\min\{\tau,T\})
\end{equation*}
and so
\begin{equation}\label{eqKalpha}
	K^{\alpha}\leq \sinh(2c_2)\max\left\{\tilde{W}(0),\beta_2(T)\right\},\quad \forall t\in[0,\min\{\tau,T\}).
\end{equation}
Next, for $t_0=\tau/2$, the estimate \eqref{s4.2-5} implies
\begin{align*}
  \tilde{W}(t)\leq &\max\left\{(\frac{n\alpha+1}{2}(t-t_0))^{-\frac{n\alpha}{n\alpha+1}}c^{-1},\beta_2(T)\right\}\nonumber\\
  \leq&~\max\left\{(\frac{(n\alpha+1)\tau}{4})^{-\frac{n\alpha}{n\alpha+1}}c^{-1},\beta_2(T)\right\}
\end{align*}
for $t\in [\tau, \min\{3\tau/2,T\})$, and so
\begin{equation}\label{s4.tdW3}
K^\alpha\leq \sinh(2c_2)\max\left\{(\frac{(n\alpha+1)\tau}{4})^{-\frac{n\alpha}{n\alpha+1}}c^{-1},\beta_2(T)\right\}
\end{equation}
for $t\in [\tau, \min\{3\tau/2,T\})$. Repeating the above argument for $t_0=m\tau/2$ ($m\geq 2$), we can get the estimate \eqref{s4.tdW3} for $t\in [\frac{(m+1)\tau}2, \min\{\frac{(m+2)\tau}2,T\})$, which covers the whole time interval $[0,T)$. Combining \eqref{eqKalpha} and \eqref{s4.tdW3}, we complete the proof of Proposition \ref{propKupp}.
\endproof

\section{Long time existence}\label{sec5}
In this section, we prove the long time existence of the flow \eqref{flow-VMCF}.
\begin{thm}\label{long}
	Let $M_0$ be a smooth closed convex hypersurface in $\mathbb{H}^{n+1}$ and $M_t$ be the smooth solution of the flow \eqref{flow-VMCF} starting from $M_0$. Then $M_t$ remains convex and exists for all time $t\in[0,\infty)$.
\end{thm}
\proof
We will argue by contradiction. Let $[0,T)$ be the maximum interval such that the solution of the flow \eqref{flow-VMCF} exists with $T<\infty$. Then combining Proposition \ref{preserve convex} and Proposition \ref{propKupp} yields that the principal curvatures $\kappa=(\kappa_1,\dots,\kappa_n)$ of the solution $M_t$ satisfy
\begin{equation}\label{s5.1}
		0<\underline{\kappa}_0\leq \kappa_i\leq \overline{\kappa}_0,\quad i=1,\dots,n
	\end{equation}
for all time $t\in [0,T)$, where the constants $\underline{\kappa}_0, \overline{\kappa}_0$ depend on $n,\alpha,M_0$ and $T$.

To prove the long time existence of the solution $M_t$ of the flow \eqref{flow-VMCF}, we need to derive the higher order regularity estimates. Recall that in $\S$\ref{subsec}, up to a tangential diffeomorphism, the flow equation \eqref{flow-VMCF} is equivalent the following scalar parabolic equation
\begin{align}\label{eqqs}
	 \frac{\partial}{\partial t}s=&A\phi(t)-B(K^{Y})^{\alpha}=~A\phi(t)-B(\det \mathfrak{r})^{-\alpha}
\end{align}
on the sphere $\mathbb{S}^n$, where $\mathfrak{r}_{ij}=s_{ij}+s\sigma_{ij}$ and $A$, $B$ are functions defined in \eqref{eqAB} which depend only on $s$ and $\bar{\nabla}s$. Combing Lemma \ref{t0-bound on u,du} and the curvature estimate \eqref{s5.1} gives the $C^2$ estimates of $s(z,t)$. Denote the right hand side of \eqref{eqqs} by $G(\bar{\nabla}^2 s,\bar{\nabla} s,s,z,t)$. Then the derivatives of $G$ with respect to the first argument are given by
\begin{equation}
	\dot{G}^{ij}=n\alpha B(\det \mathfrak{r})^{-\alpha-\frac{1}{n}}\frac{\partial(\det \mathfrak{r})^{\frac{1}{n}}}{\partial s_{ij}}
\end{equation}
and
\begin{equation}
\begin{split}
	\ddot{G}^{ij,k\ell}&=n\alpha B(\det \mathfrak{r})^{-\alpha-\frac{1}{n}}\frac{\partial^2(\det \mathfrak{r})^{\frac{1}{n}}}{\partial s_{ij}\partial s_{k\ell}}\\
	&-n\alpha(n\alpha+1)B(\det \mathfrak{r})^{-\alpha-\frac{2}{n}}\frac{\partial(\det \mathfrak{r})^{\frac{1}{n}}}{\partial s_{ij}}\frac{\partial(\det \mathfrak{r})^{\frac{1}{n}}}{\partial s_{k\ell}}.
\end{split}
\end{equation}

The estimates we have established imply the existence of a constant $C=C(T)>0$, such that $1/C I\leq (\dot{G^{ij}})\leq CI$, that is, the operator $G$ is uniformly elliptic on the finite time interval $[0,T)$. Since $(\det \mathfrak{r})^{\frac{1}{n}}$ is concave with respect to $s_{ij}$, we see that $G$ is also a concave operator. We can apply Theorem 1.1 in \cite{TW13} to obtain a $C^{2,\gamma}$ estimate on $s$, for a suitable $\gamma\in(0,1)$. See also the arguments in \cite{CS10,Mc05} for the $C^{2,\gamma}$ estimate of the solutions to volume preserving curvature flows. Then by the parabolic Schauder theory (see \cite{Lie96}), we can deduce all higher order regularity estimates of $s$ on $[0,T)$ and a standard continuation argument then shows that $T=+\infty$.
\endproof
\begin{rem}
Note that the curvature estimate \eqref{s5.1} of the solution $M_t$ of the flow \eqref{flow-VMCF} depends on time $t$ and may degenerate as time $t\to\infty$. To study the asymptotical behavior of $M_t$ as $t\to\infty$, we still need to get an uniform curvature estimate which does not depend on time. This will be obtained in the next two sections.
\end{rem}

\section{Monotonicity and Hausdorff convergence}\label{sec.hau}
In this section, we prove the monotonicity of $\mathcal{A}_{n-1}(\Omega_t)$, the subsequential Hausdorff convergence of the solution $M_t$ of \eqref{flow-VMCF} and the convergence of the center of the inner ball of $\Omega_t$ to a fixed point.

Denote the average integral of the Gauss curvature by
\begin{equation}\label{s6.0}
	\bar{K}=\frac{1}{|M_t|}\int_{M_t}{K d\mu_t}=\frac{\mathcal{A}_n(\Omega_t)+\frac{1}{n-1}\mathcal{A}_{n-2}(\Omega_t)}{\mathcal{A}_0(\Omega_t)}.
\end{equation}
It follows from the monotonicity \eqref{s2.Akmo} of quermassintegrals with respect to inclusion of convex sets and the estimates on inner radius and outer radius in Lemma \ref{in-out} that $\bar{K}$ is uniformly bounded from above and below by positive constants depending only on $n$, $\alpha$ and $M_0$.
\subsection{Monotonicity for $\mathcal{A}_{n-1}$}
We first show the following monotonicity of $\mathcal{A}_{n-1}(\Omega_t)$ along the flow \eqref{flow-VMCF}, which will be useful in proving the subsequential Hausdorff convergence of $M_t$.
\begin{lem}\label{lemmono}
	Let $M_t$ be a smooth convex solution of the volume preserving flow \eqref{flow-VMCF}. Denote by $\Omega_t$ the domain enclosed by $M_t$. Then $\mathcal{A}_{n-1}(\Omega_t)$ is monotone decreasing in time $t$, and is strictly decreasing unless $\Omega_t$ is a geodesic ball.
\end{lem}
\proof
From the evolution equation \eqref{eqWk} for the quermassintegrals of $\Omega_t$, we have
\begin{equation*}
	\frac{d}{dt}\mathcal{A}_{n-1}(\Omega_t)=n\int_{M_t}{K(\phi(t)-K^{\alpha})d\mu_t}.
\end{equation*}
Since $\phi(t)$ is defined as in \eqref{eqphi}, we have
\begin{align}\label{eqWmo}
	\frac{d}{dt}\mathcal{A}_{n-1}(\Omega_t)&=\frac{n}{|M_t|}\left(\int_{M_t}K d\mu_t\int_{M_t}K^{\alpha} d\mu_t-|M_t|\int_{M_t}K^{\alpha+1} d\mu_t\right)\notag\\
	&=-n\int_{M_t}{(K-\bar{K})(K^{\alpha}-{\bar{K}}^{\alpha})d\mu_t}\leq 0.
\end{align}
Note that equality holds in \eqref{eqWmo} if and only if $K$ is a constant on $M_t$, which means $M_t$ is a geodesic sphere by the Alexandrov type theorem for hypersurfaces with constant Gauss curvature in the hyperbolic space (see \cite{MS91}).
\endproof
\subsection{Subsequential Hausdorff convergence}
In this subsection, we prove that there exists a sequence of times $t_i\to \infty$ such that the solution $M_{t_i}$ converges in Hausdorff sense to a geodesic sphere. We first prove the following estimate:
\begin{lem}
Let $M_0$ be a smooth closed and convex hypersurface in $\mathbb{H}^{n+1}$ and $M_t$ be the smooth solution of the flow \eqref{flow-VMCF} starting from $M_0$. Then there exists a sequence of times $\{t_i\},t_i\to\infty$, such that
\begin{equation}\label{s6.2-1}
  \int_{M_{t_i}}{|K-\bar{K}|d\mu_{t_i}}\to 0,\quad \text{as}\,\, t_i\to\infty.
\end{equation}
\end{lem}
\proof
Applying the monotonicity in Lemma \ref{lemmono} and the long time existence of the flow \eqref{flow-VMCF}, we have
\begin{equation*}
	n\int_0^{\infty}\int_{M_t}{(K-\bar{K})(K^{\alpha}-{\bar{K}}^{\alpha})d\mu_t}dt\leq {\mathcal{A}_{n-1}}(\Omega_0)<\infty.
\end{equation*}
Therefore there exists a sequence of times $t_i\to\infty$ such that
\begin{equation}\label{s6.2-2}
	\int_{M_{t_i}}{(K-\bar{K})(K^{\alpha}-{\bar{K}}^{\alpha})d\mu_{t_i}}\to 0.
\end{equation}
If $\alpha\geq 1$, we have
\begin{align*}
	(K-\bar{K})(K^{\alpha}-{\bar{K}}^{\alpha})&=\alpha\int_0^1{((1-s)\bar{K}+sK)^{\alpha-1}ds}\cdot (K-\bar{K})^2\\
	&\geq \alpha\int_0^1{(1-s)^{\alpha-1}ds}\bar{K}^{\alpha-1}(K-\bar{K})^2\\
	&\geq C(K-\bar{K})^2,
\end{align*}
where we used the fact that $\bar{K}$ is uniformly bounded and $\alpha\geq 1$. Therefore, \eqref{s6.2-2} implies that
\begin{equation*}
	\int_{M_{t_i}}{(K-\bar{K})^2d\mu_{t_i}}\to 0,\quad \text{as}\,\, i\to\infty
\end{equation*}
for a sequence of times $t_i\to\infty$ and the estimate \eqref{s6.2-1} follows by the H\"{o}lder inequality.

We next prove the estimate \eqref{s6.2-1} for $0<\alpha<1$. In this case,  we have
\begin{align*}
  &\int_{M_{t_i}}|K-\bar{K}|d\mu_{t_i} \\
  =&\int_{M_{t_i}}|K-\bar{K}|^{1/2}|K-\bar{K}|^{1/2}d\mu_{t_i} \\
  =&  \int_{M_{t_i}}\Big(\frac 1{\alpha} \int_0^1{((1-s)\bar{K}^\alpha+sK^\alpha)^{\frac{1}{\alpha}-1}ds}\Big)^{1/2}\cdot \Big|(K^\alpha-\bar{K}^\alpha)(K-\bar{K})\Big|^{1/2}d\mu_{t_i} \\
   \leq & \Big(\underbrace{\int_{M_{t_i}}\frac 1{\alpha} \int_0^1{((1-s)\bar{K}^\alpha+sK^\alpha)^{\frac{1}{\alpha}-1}ds}d\mu_{t_i}}_{(I)}\Big)^{1/2}\Big(\underbrace{\int_{M_{t_i}}(K-\bar{K})(K^{\alpha}-{\bar{K}}^{\alpha})d\mu_{t_i}}_{(II)}\Big)^{1/2}.
\end{align*}
The second term $(II)$ tends to zero as $t_i\to\infty$ by \eqref{s6.2-2}. We show that the first term $(I)$ is bounded for any $0<\alpha<1$. In fact, if $\alpha\in [\frac 12,1)$, we have $1<1/\alpha\leq 2$. This implies that
\begin{align*}
  (I)=  & \frac 1{\alpha} \int_{M_{t_i}}\int_0^1{((1-s)\bar{K}^\alpha+sK^\alpha)^{\frac{1}{\alpha}-1}ds}d\mu_{t_i} \\
  \leq & \frac 1{\alpha} \int_{M_{t_i}}(\bar{K}^\alpha+K^\alpha)^{\frac{1}{\alpha}-1}d\mu_{t_i}\\
  =& \frac 1{\alpha} \int_{M_{t_i}}\frac{\bar{K}^\alpha+K^\alpha}{(\bar{K}^\alpha+K^\alpha)^{2-\frac{1}{\alpha}}}d\mu_{t_i}\\
  \leq &\frac 1{\alpha\bar{K}^{2\alpha-1}} \int_{M_{t_i}}(\bar{K}^\alpha+K^\alpha)d\mu_{t_i}\\
  \leq & \frac 2{\alpha} \bar{K}^{1-\alpha}\mathcal{A}_0(\Omega_{t_i})\leq ~C
\end{align*}
is uniformly bounded from above, where we used the H\"{o}lder inequality in the fourth inequality to get the estimate
\begin{equation*}
 \int_{M_{t_i}}K^\alpha d\mu_{t_i}\leq \left(\int_{M_{t_i}}Kd\mu_{t_i}\right)^\alpha \mathcal{A}_0(\Omega_{t_i})^{1-\alpha}=\bar{K}^{\alpha}\mathcal{A}_0(\Omega_{t_i})
\end{equation*}
for $\alpha<1$ and that $\mathcal{A}_0(\Omega_{t_i})\leq \mathcal{A}_0(B_{c_2}(0))\leq C$ which is due to that the outer radius of $\Omega_{t_i}$ is bounded from above by $c_2$. We repeat the argument for $\alpha\in [\frac{1}{k+1},\frac{1}{k})$ with $k=2,3,\cdots$, where we have $(k+1)\alpha\geq 1$ and $0<k\alpha<1$. Then
\begin{align*}
  (I) \leq & \frac 1{\alpha} \int_{M_{t_i}}(\bar{K}^\alpha+K^\alpha)^{\frac{1}{\alpha}-1}d\mu_{t_i}\\
  =& \frac 1{\alpha} \int_{M_{t_i}}\frac{(\bar{K}^\alpha+K^\alpha)^k}{(\bar{K}^\alpha+K^\alpha)^{k+1-\frac{1}{\alpha}}}d\mu_{t_i}\\
  \leq &\frac 1{\alpha\bar{K}^{(k+1)\alpha-1}} \int_{M_{t_i}}2^{k-1}(\bar{K}^{k\alpha}+K^{k\alpha})d\mu_{t_i}\\
  \leq & \frac {2^k}{\alpha} \bar{K}^{1-\alpha}\mathcal{A}_0(\Omega_{t_i})\leq ~C
\end{align*}
is uniformly bounded, where in the third inequality we used the H\"{o}lder inequality again as $k\alpha<1$. Therefore for any $0<\alpha<1$, the term $(I)$ is uniformly bounded from above and thus the estimate \eqref{s6.2-1} follows in this case.
\endproof

Note that the curvature estimate \eqref{s5.1} of the solution $M_t$ of the flow \eqref{flow-VMCF} depends on time $t$ and may degenerate as time $t\to\infty$. So we can not conclude from \eqref{s6.2-1} that the solution $M_{t_i}$ converges to a geodesic sphere as $t_i\to\infty$, as we do not have an uniform regularity estimate to guarantee the existence of a smooth limit of $M_{t_i}$. However, we can apply the similar idea in the work \cite[\S 6]{AW21} by the first author and Andrews in the Euclidean space to prove the following subsequential Hausdorff convergence result:
\begin{lem}\label{subcon}
	Let $M_0$ be a smooth, closed convex hypersurface in $\mathbb{H}^{n+1}$ and $M_t$ be the smooth solution of the flow \eqref{flow-VMCF} starting from $M_0$. Then there exists a sequence of times $\{t_i\},t_i\to\infty$, such that $M_{t_i}$ converges to a geodesic sphere $S_{\rho_{\infty}}(p)$ in Hausdorff sense as $t_i\to\infty$, where $p$ is the center of the sphere and the radius $\rho_{\infty}$ is determined by the fact that $S_{\rho_{\infty}}(p)$ encloses the same volume of $M_0$.
\end{lem}
\proof
Let $p_t$ be the center of the inball of $\Omega_t$ and let $\varphi_t:\mathbb{H}^{n+1}\to\mathbb{H}^{n+1}$ be an isometry carrying $p_t$ to the origin $o\in \mathbb{H}^{n+1}$. Clearly, each $\varphi_t(M_t)$ is a closed convex hypersurface with an inball centered at the origin and having inner radius $\rho_{-}(t)\geq c_1$.  For simplicity of the notations, we still denote the transformed solution $\varphi_t(M_t)$ as $M_t$.

As in \S \ref{subsec}, we project $\Omega_t\subset \mathbb{H}^{n+1}$ onto $B_1(0)\subset\mathbb{R}^{n+1}$ and get the corresponding $\tilde{\Omega}_t\subset B_1(0)\subset \mathbb{R}^{n+1}$. Since the outer radius of $\Omega_t$ is uniformly bounded above and so does $\tilde{\Omega}_t$, the Blaschke selection theorem (see Theorem 1.8.7 of \cite{RS2014}) implies that there exists a sequence of times $t_i$ and a convex body $\tilde{\Omega}$ such that $\tilde{\Omega}_{t_i}$ converges to $\tilde{\Omega}$ in Hausdorff sense as $t_i\to\infty$. Note that the projection yields an one-to-one correspondence between the convex bodies in $\mathbb{H}^{n+1}$ and the convex bodies in $B_1(0)\subset \mathbb{R}^{n+1}$, then there exists a convex set $\hat{\Omega}\in\mathcal{K}(\mathbb{H}^{n+1})$ such that $\Omega_{t_i}$ converges to $\hat{\Omega}$ in Hausdorff sense as $t_i\to\infty$. As each $\Omega_{t_i}$ has inner radius $\rho_{-}(\Omega_{t_i})\geq c_1$, the limit convex set $\hat{\Omega}$ has positive inner radius. Without loss of generality, we may assume that the sequence $t_i$ is the same sequence such that \eqref{s6.2-1} holds.

By the continuity of quermassintegrals with respect to the Hausdorff distance in Lemma \ref{inquer}, the Hausdorff convergence of $\Omega_{t_i}$ to $\hat{\Omega}$ implies that
\begin{equation}\label{s6.1-0}
  \bar{K}=\frac{\mathcal{A}_n(\Omega_{t_i})+\frac{1}{n-1}\mathcal{A}_{n-2}(\Omega_{t_i})}{\mathcal{A}_0(\Omega_{t_i})}~\to~\frac{\mathcal{A}_n(\hat{\Omega})+\frac{1}{n-1}\mathcal{A}_{n-2}(\hat{\Omega})}{\mathcal{A}_0(\hat{\Omega})}~=:c
\end{equation}
as $t_i\to \infty$.  We will show that $\hat{\Omega}$ satisfies the equation $\Phi_0(\hat{\Omega},\cdot)=c\Phi_n(\hat{\Omega},\cdot)$ for the curvature measures $\Phi_0$ and $\Phi_n$. In fact, by the weak continuity of the curvature measures in Theorem \ref{s2.thmcurv}, for any bounded continuous function $f$ on $\mathbb{H}^{n+1}$ with compact support, we have that $\int{f d\Phi_0(\Omega_{t_i})}$ converges to $\int{f d\Phi_0(\hat{\Omega})}$, and $\int{f d\Phi_n(\Omega_{t_i})}$ converges to $\int{f d\Phi_n(\hat{\Omega})}$, as $i\to\infty$. Then
\begin{align}\label{s6.1-1}
	&\left|\int{f d\Phi_0(\Omega_{t_i})}-c\int{f d\Phi_n(\Omega_{t_i})}\right|\nonumber\\
	=&\left|\int_{M_{t_i}}{fK d\mathcal{H}^n}-\int_{M_{t_i}}{cf d\mathcal{H}^n}\right|\nonumber\\
	\leq&\sup|f|\int_{M_{t_i}}{|K-c| d\mathcal{H}^n}\nonumber\\
	\leq&\sup|f|\int_{M_{t_i}}{|K-\bar{K}| d\mathcal{H}^n}+\sup|f|\int_{M_{t_i}}{|\bar{K}-c| d\mathcal{H}^n}~\to~0
\end{align}
by the estimates \eqref{s6.2-1} and \eqref{s6.1-0}.

It follows that $\int{f d\Phi_0(\hat{\Omega})}=c\int{f d\Phi_n(\hat{\Omega})}$ for all bounded continuous functions $f$, and therefore $\Phi_0(\hat{\Omega},\cdot)=c\Phi_n(\hat{\Omega},\cdot)$ as claimed. Then by Theorem \ref{gAT}, $\hat{\Omega}$ is a geodesic ball. This means that $\varphi_{t_i}(\Omega_{t_i})$ converges to a geodesic ball of some radius $\rho_\infty$ centered at the origin. The radius is uniquely determined by the preserving of the volume. Moreover, by an argument of Alexandrov reflection, the evolving domains $\Omega_t$ cannot leave away from the initial domain $\Omega_0$, that is, there exists a geodesic ball $B_R\subset \mathbb{H}^{n+1}$ of radius $R>0$ containing $\Omega_0$ such that $\Omega_t\cap B_R\neq \emptyset$ for all time $t\in [0,\infty)$, see Lemma 4.1 in \cite{BenChenWei}. In particular, this implies that the centers $p_{t_i}$ of inner ball of $\Omega_{t_i}$ are located in a compact subset of $\mathbb{H}^{n+1}$, and then there exists a subsequence of $t_i$ (still denoted by $t_i$) such that $p_{t_i}$ converges to a limit point $p\in \mathbb{H}^{n+1}$. This concludes that for this subsequence of times $t_i$, $\Omega_{t_i}$ converges to a geodesic ball $B_{\rho_\infty}(p)$ centered at the point $p$ without correction of ambient isometry $\varphi_{t_i}$. This completes the proof of Lemma \ref{subcon}.
\endproof

\subsection{Convergence of the center of the inner ball}
The argument in \cite{AW21} is not sufficient for us to deduce the Hausdorff convergence of $M_t$ to the geodesic sphere for all time $t\to \infty$, as we do not have the analogue stability estimate as in \cite[Eq.(7.124)]{RS2014} for the hyperbolic case. However, if we denote $p_t$ as the center of the inner ball of $\Omega_t$, we can prove that $p_t$ converges to the fixed point $p\in \mathbb{H}^{n+1}$ for all time $t\to\infty$ using the Alexandrov reflection and the subsequential Hausdorff convergence of $M_t$ proved in Lemma \ref{subcon}.

Recall that $p\in \mathbb{H}^{n+1}$ is the center of the limit geodesic sphere $S_{\rho_{\infty}}(p)$ in Lemma \ref{subcon}. Take an arbitrary direction $z\in T_p\mathbb{H}^{n+1}$. Let $\gamma_z$ be the normal geodesic line (i.e. $|\gamma'|=1$) through the point $p$ with $\gamma_z(0)=p$ and $\gamma'_z(0)=z$, and let $H_{z,s}$ be the totally geodesic hyperplane in $\mathbb{H}^{n+1}$ that is perpendicular to $\gamma_z$ at $\gamma_z(s), s\in\mathbb{R}$. We use the notation $H_{z,s}^{+}$ and $H_{z,s}^{-}$ for the half-spaces in $\mathbb{H}^{n+1}$ determined by $H_{z,s}$ as follows:
\begin{equation*}
	H_{z,s}^{+}:=\bigcup_{s'\geq s}H_{z,s'},\qquad  H_{z,s}^{-}:=\bigcup_{s'\leq s}H_{z,s'}.
\end{equation*}
For a bounded domain $\Omega$ in $\mathbb{H}^{n+1}$, denote
\begin{equation*}
	\Omega_z^{+}(s)=\Omega\cap H_{z,s}^{+},\qquad \Omega_z^{-}(s)=\Omega\cap H_{z,s}^{-}.
\end{equation*}
The reflection map across $H_{z,s}$ is denoted by $R_{\gamma_z,s}$. We define
\begin{align*}
	S_{\gamma_z}^{+}(\Omega)&:=\inf\{s\in\mathbb{R}~|~R_{\gamma_z,s}(\Omega_z^{+}(s))\subset\Omega_z^{-}(s)\},\\
	S_{\gamma_z}^{-}(\Omega)&:=\sup\{s\in\mathbb{R}~|~R_{\gamma_z,s}(\Omega_z^{-}(s))\subset\Omega_z^{+}(s)\}.
\end{align*}
The Alexandrov reflection argument implies that $S_{\gamma_z}^{+}(\Omega_t)$ is non-increasing in $t$ for each $z$ (see \cite[Lemma 6.1]{BenWei}). By the definitions of $S_{\gamma_z}^{+}(\Omega_t)$ and $S_{\gamma_z}^{-}(\Omega_t)$, we have $S_{\gamma_z}^{-}(\Omega_t)\leq S_{\gamma_z}^{+}(\Omega_t)$. Since $S_{\gamma_z}^{-}(\Omega_t)=-S_{\gamma_{-z}}^{+}(\Omega_t)$, we also have that $S_{\gamma_z}^{-}(\Omega_t)$ is non-decreasing in $t$ for each $z$. Note that the paper \cite{BenWei} deals with the flow with $h$-convex initial hypersurfaces, the argument in Lemma 6.1 of \cite{BenWei} works for convex solutions as well. The readers may refer to \cite{Chow97,CG96} for more details on the Alexandrov reflection method.

We need the following lemma in proving the convergence of the center of the inner ball.
\begin{lem}\label{center}
	Let $\Omega$ be a bounded convex domain in $\mathbb{H}^{n+1}$, and $\gamma_z, H_{z,s}, S_{\gamma_z}^{+}(\Omega)$ and $S_{\gamma_z}^{-}(\Omega)$ be defined as above. Denote $p_0$ as the center of an inner ball of $\Omega$ and assume that $p_0\in H_{z,s_0}$. Then we have $S_{\gamma_z}^{-}(\Omega)\leq s_0\leq S_{\gamma_z}^{+}(\Omega)$.
\end{lem}
\proof
Denote $S_{\gamma_z}^{+}(\Omega)$ by $\bar{s}$, then we have $R_{\gamma_z,\bar{s}}(\Omega_z^{+}(\bar{s}))\subset\Omega_z^{-}(\bar{s})$. Assume that $B_{r_0}(p_0)$ is the inner ball of $\Omega$ centered at $p_0$. First we prove that $s_0\leq \bar{s}$. If not (i.e. $s_0>\bar{s}$), then $B_{r_0}(p'_0):=R_{\gamma_z,\bar{s}}(B_{r_0}(p_0))\subset\Omega$ by $R_{\gamma_z,\bar{s}}(\Omega_z^{+}(\bar{s}))\subset\Omega_z^{-}(\bar{s})$ and the assumption that $s_0>\bar{s}$. See Figure \ref{figure center}. Here $p'_0=R_{\gamma_z,\bar{s}}(p_0)\subset H_{z,s'_0}$ with $s'_0<\bar{s}<s_0$. Then we can deduce that the reflection ball $B_{r_0}(p'_0)\cap \partial\Omega=\emptyset$. It follows that $B_{r_0+\varepsilon}(p'_0)\subset\Omega$ for sufficiently small $\varepsilon$ and hence $p_0$ cannot be the center of an inner ball, which leads to a contradiction.
\begin{figure}[h]
	\begin{tikzpicture}[scale=1.5]
		\path[shade,thick,draw] (-3.5,-3.4) [rounded corners=20pt] -- (-4.3,-2.6)
		[rounded corners=20pt] -- (-4,-1.3)
		[rounded corners=20pt]--(-2.8,-0.1)
		[rounded corners=10pt]--(-1.7,-0.3)
		[rounded corners=20pt]--(-0.3,-0.7)
		[rounded corners=20pt]--(-0.1,-1.9)
		[rounded corners=10pt]--(-0.4,-2.7)
		[rounded corners=20pt]--(-1.9,-3.7)
		[rounded corners=20pt]--cycle;
		\draw[->][thick](-5,-2.4)--(2,-2.4)node[below]{$\gamma_{z}(s)$};
		\draw[->][thick](-3,-4)--(-3,0.6)node[right]{$H_{z,0}$};
		\draw[blue,thick](-1.5,-4)--(-1.5,0.5)node[left]{$H_{z,\bar{s}}$};
		\draw[red,thick](-1.3,-4)--(-1.3,0.5)node[right]{$H_{z,s_0}$};
		\draw[thick](-1.3,-1.5) circle(1.03);
		\filldraw[thick](-1.3,-1.5) circle(0.03) node[right]{$p_0$};
		\draw[dashed](-1.7,-1.5) circle(1.03);
		\filldraw[thick](-1.7,-1.5) circle(0.03) node[left]{$p'_0$};
		\filldraw[thick](-3,-2.4) circle(0.03);
		\node at (-3.1,-2.55) {$p$};
		\node at (-4,0.3){$\mathbb{H}^{n+1}$};
		\node[blue] at (-3.5,-3){$\Omega$};
	\end{tikzpicture}
	\caption{Center of the inner ball.}\label{figure center}
\end{figure}
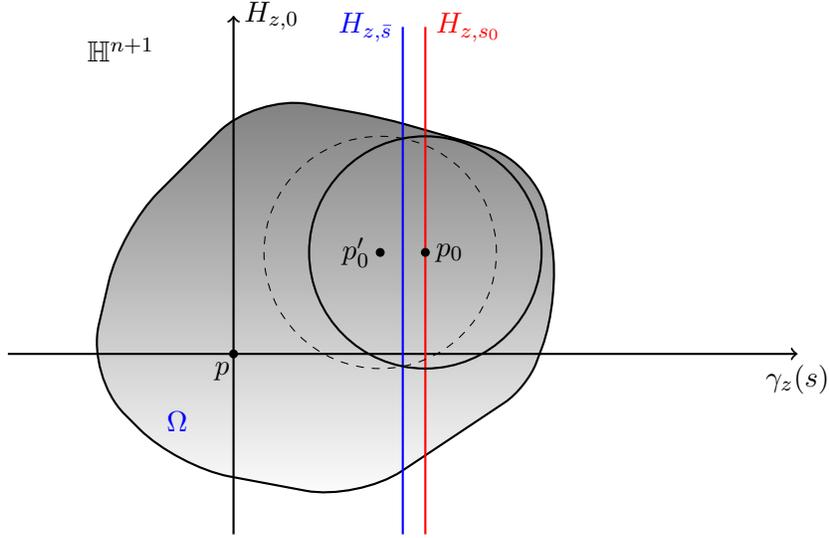

The inequality $S_{\gamma_z}^{-}(\Omega)\leq s_0$ can be checked similarily as above. This completes the proof of Lemma \ref{center}.
\endproof

\begin{rem}
For a bounded domain $\Omega$, the center of the inner ball may not be unique. Lemma \ref{center} says that all of them must satisfy the two-sided bounds.
\end{rem}

Then we can prove the following convergence result:
\begin{lem}\label{innercon}
Let $M_0$ be a smooth, closed convex hypersurface in $\mathbb{H}^{n+1}$ and $M_t$ be the smooth solution of the flow \eqref{flow-VMCF} starting from $M_0$.	Denote the enclosed domain of $M_t$ by $\Omega_t$. Take an arbitrary direction $z\in T_p{\mathbb{H}^{n+1}}$ and let $\gamma_z, S_{\gamma_z}^{-}(\Omega_t), S_{\gamma_z}^{+}(\Omega_t)$ be defined as above. Then along the flow \eqref{flow-VMCF}, we have
	\begin{equation}\label{reflection}
		\lim_{t\to\infty}S_{\gamma_z}^{-}(\Omega_t)=\lim_{t\to\infty}S_{\gamma_z}^{+}(\Omega_t)=0.
	\end{equation}
	As a consequence, if we set $p_t$ as the center of an inner ball of $\Omega_t$, then we have $d(p_t,p)\to 0$ as $t\to\infty$.
\end{lem}
\begin{proof}
Recall that in Lemma \ref{subcon}, we have proved that there exists a sequence of times $t_i\to\infty$, such that $M_{t_i}$ converges to a geodesic sphere $S_{\rho_{\infty}}(p)=\partial B_{\rho_{\infty}}(p)$ in Hausdorff sense as $t_i\to\infty$, where $\rho_{\infty}$ is the radius determined by $\Omega_0$ and $p$ is the center of the ball. Then there exists a sequence of  $d_i\rightarrow 0$ with $d_i<\rho_{\infty}$ such that
	\begin{equation*}
	M_{t_i}\subset B_{\rho_{\infty}+d_i}(p)/ B_{\rho_{\infty}-d_i}(p).
	\end{equation*}	
	Hence by the monotonicity of $S_{\gamma_z}^{+}(\Omega_t)$, if we can show that $S_{\gamma_z}^{+}(\Omega_{t_i})\leq C\sqrt{d_i}$ for some $C=C(\rho_{\infty})>0$, then we have $S_{\gamma_z}^{+}(\Omega_{t})\leq C\sqrt{d_i},\ \forall t\geq t_i$.
	
Denote the geodesic segment starting from $A_1$ to $A_2$ by $\overline{A_1A_2}$, and the length by $|\overline{A_1A_2}|$. 	Let $P,\ P'$ be two totally geodesic hyperplane which is perpendicular to $\gamma_{z}$ and passes through $p$ and some $p'$ respectively on $\gamma_{z}$. Let $E=P\cap\partial B_{\rho_\infty-d_i}(p)$ and $E'=P'\cap\partial B_{\rho_\infty+d_i}(p)$. By a simple continuity argument, there exists $P'$ such that $\forall e'\in E'$, the geodesic line starting form $e'$ and is perpendicular to $P'$ would pass through some $e\in E$, see figure \ref{fig2}.

\begin{figure}[h]
	\begin{tikzpicture}[scale=1.2]
		\path[shade,draw,thick] (-4.2,-3.6) [rounded corners=11pt] -- (-4.75,-2.6)
		[rounded corners=28pt]-- (-4.6,-0.8)
		[rounded corners=20pt]--(-2.8,0.4)
		[rounded corners=10pt]--(-1.7,0.2)
		[rounded corners=20pt]--(-0.3,-0.5)
		[rounded corners=15pt]--(0.1,-1.9)
		[rounded corners=15pt]--(-0.4,-3.2)
		[rounded corners=25pt]--(-1.9,-4.3)
		[rounded corners=25pt]--cycle;
		\draw[dashed,blue](-2.5,-1.9) circle(2);
		\draw[dashed,red](-2.5,-1.9) circle(2.7);
		\draw[thick](-2.5,0.1)--(-0.59,0.01);
		\filldraw[thick](-2.5,0.1) circle(0.03);
		
		\filldraw[thick](-0.59,0.01) circle(0.03);
			\node at (-0.05,0.2) {$e'\in E'$};
		
		\filldraw[thick](-2.5,-1.9) circle(0.03);
		\node at (-2.3,-0.1) {$e\in E$};
	
		\node at (-2.3,-2.1) {$p$};
		
		\draw[thick] (-0.65,-1.9) arc (180:186:30);
		\draw[thick] (-0.65,-1.9) arc (180:174:30);
		
		\draw[->][thick](-6.0,-1.9)--(2.0,-1.9) node[below]{$\gamma_z$};
		\draw[thick](-2.5,-5.1)--(-2.5,1.2) node[right]{$P$};
		\node at (-0.2,1.17) {$P'$};
		
		\node at (-1.36,-5) {$H_{z,s}^-$};
		\node at (0,-5) {$H_{z,s}^+$};

		\draw[red,thick](-4.35,-1.15)--(0.14,-1.35);

		\filldraw[thick](-4.35,-1.15) circle(0.03);
			\filldraw[thick](0.14,-1.35) circle(0.03);
		\filldraw[thick](-2.43,-1.23) circle(0.03);
		
			\draw[thick,dashed](-2.5,-1.9)--(-2.43,-1.23);
		
		\filldraw[thick](-0.65,-1.32) circle(0.03);
		\filldraw[thick](-0.2,-1.33) circle(0.03);
		
		\filldraw[thick](-0.65,-1.9) circle(0.03);
	
		\node at (-4.25,-1.7) {$l$};
		\node at (-2.3,-1) {$a$};
		\node at (-0.5,-1.1) {$b$};
		\node at (-0.25,-1.1) {$A$};
		\node at (0.35,-1.15) {$r$};
		\node at (-0.85,-2.1) {$p'$};
		\filldraw[thick](-1.16,-1.3) circle(0.03);
		\node at (-1.3,-1) {$R_{\gamma_z,s}(A)$};
		\draw[->][thick](-1.3,-2.3)--(-0.8,-2.3);
		\draw[->][thick](-1.9,-2.3)--(-2.4,-2.3);
		\node at (-1.6,-2.3) {$s$};
		\node at (-5,1.1){$\mathbb{H}^{n+1}$};
		\node[blue] at (-3.5,-3){$\Omega$};
		
		\draw (-2.32,-1.25)--(-2.33,-1.36)--(-2.43,-1.35);
		
		\draw (-0.65,-1.8)--(-0.75,-1.8)--(-0.75,-1.9);

		\draw (-0.75,-1.32)--(-0.75,-1.42)--(-0.65,-1.42);

			\draw (-0.69,0.01)--(-0.69,-0.09)--(-0.59,-0.09);
	
	\end{tikzpicture}
	\caption{$R_{\gamma_z,s}(A)\in\Omega_z^-(s)$}\label{fig2}
\end{figure}
Furthermore, we can pick $P'$ such that 
\begin{equation}
	d(P,P')=|\overline{pp'}|\leq O(\sqrt{d_i}).
\end{equation}
To see this, by the law of sines and cosines in hyperbolic space, denote $\angle A_1A_2A_3$ as the angle between $\overline{A_1A_2}$ and $\overline{A_2A_3}$, we have
\begin{equation}\label{pythargorean}
	\left\{\begin{aligned}
		&\cosh|\overline{pp'}| \cosh|\overline{e'p'}|=\cosh(\rho_\infty+d_i),\\
		&\cosh|\overline{e'p'}| \cosh|\overline{ee'}|=\cosh(\rho_\infty-d_i) \cosh|\overline{pp'}|\left(=\cosh|\overline{ep'}|\right)\\
		&\frac{\sinh|\overline{e'p'}|}{\sinh(\rho_\infty+d_i)}=\sin\angle e'pp'=\cos\angle epe'=\frac{\cosh(\rho_\infty-d_i)\cosh(\rho_\infty+d_i)-\cosh|\overline{ee'}|}{\sinh(\rho_\infty-d_i)\sinh(\rho_\infty+d_i)}.
	\end{aligned}\right.
\end{equation}
Solving \eqref{pythargorean}, we get 
\begin{equation*}
	\sinh(\rho_{\infty}-d_i)\left(1+\sinh^2|\overline{e'p'}|\right)=\cosh(\rho_\infty-d_i)\cosh(\rho_\infty+d_i)\sinh|\overline{e'p'}|.
\end{equation*}
We see that there is a solution with $\sinh|\overline{e'p'}|=\sinh\rho_\infty+O(d_i)$ and hence $\cosh|\overline{pp'}|=1+O(d_i),\ |\overline{pp'}|=O(\sqrt{d_i})$.

Next, we claim that $S_{\gamma_z}^+(\Omega_{t_i})\leq |\overline{pp'}|$, which yields
 $\lim_{t\to\infty}S_{\gamma_z}^+(\Omega_{t})\leq 0$. To see this, we have to prove that for $s:=|\overline{pp'}|$, $R_{\gamma_{z},s}(\Omega^+_{t_i}(s))\subset\Omega^-_{t_i}(s)$.
 
 Note that $\Omega^+_{t_i}(s)$ is contained in the cylinder region whose boundary corresponds to the union of normal geodesic lines through $e'\in E'$ in direction perpendicular to $P'$. Hence $\forall A\in\Omega_{t_i,z}^+(s)$, the geodesic line which goes through $A$ and is perpendicular to $P'$, would intersect $\partial B_{\rho_{\infty}+d_i}(p)\cap H_{z,s}^+$, $P'$ and $\partial B_{\rho_{\infty}-d_i}(p)\cap H_{z,s}^-$ at some points $r$, $b$, $l$ respectively. Let $a$ be the closest point on $\overline{rl}$ to $p$, see figure \ref{fig2}. By the convexity of $\Omega_{t_i}$ and that $B_{\rho_{\infty}-d_i}(p)\subset\Omega_{t_i}\subset B_{\rho_{\infty}+d_i}(p)$, we have
 \begin{equation*}
 	\overline{lb}\subset \Omega_{t_i,z}^-(s)
 \end{equation*}
 and $|\overline{bA}|\leq |\overline{br}|$. For our purpose, if we can show
 \begin{equation*}
 	|\overline{lb}|>|\overline{br}|,
 \end{equation*}
then the reflection argument follows. 

Denote $\beta=|\overline{bp'}|$. Again by the law of sines and cosines in hyperbolic space, one can calculate
 \begin{equation*}
	\tanh|\overline{pp'}|=\tanh|\overline{ab}|\cosh\beta,
\end{equation*} 
hence $|\overline{ab}|$ is monotonously decreasing in $\beta$. It's obvious that $|\overline{ap}|$ is monotonously increasing in $\beta$, thus
\begin{equation*}
	|\overline{la}|-|\overline{ar}|=\text{arccosh}\left(\frac{\cosh(\rho_\infty-d_i)}{\cosh|\overline{ap}|}\right)-\text{arccosh}\left(\frac{\cosh(\rho_\infty+d_i)}{\cosh|\overline{ap}|}\right)
\end{equation*} 
is monotonously decreasing in $\beta$. Then
\begin{equation}
	|\overline{lb}|-|\overline{br}|=(|\overline{la}|-|\overline{ar}|)+2|\overline{ab}|\geq\left(|\overline{lb}|-|\overline{br}|\right)\bigg|_{\beta=|\overline{p'e'}|}=|\overline{ee'}|>0.
\end{equation}
 
The same argument can also show $\lim_{t\to\infty}S_{\gamma_z}^-(\Omega_{t})\geq 0$. Then
	\begin{equation*}
	0\leq \lim_{t\to\infty}S_{\gamma_z}^-(\Omega_{t_i})\leq\lim_{t\to\infty}S_{\gamma_z}^+(\Omega_{t_i})\leq 0,
	\end{equation*}
	which finishes our proof.
\end{proof}

\section{Proofs of Theorem \ref{theo} and Corollary \ref{coro}}\label{final}

In this section, we complete the proofs of Theorem \ref{theo} and Corollary \ref{coro}.
\subsection{Proof of Theorem \ref{theo}}
Firstly, we prove the following uniform estimate for the principal curvatures of $M_t$ along the flow \eqref{flow-VMCF}.
\begin{lem}\label{uni}
	Let $M_0$ be a smooth, closed and convex hypersurface in $\mathbb{H}^{n+1} (n\geq 2)$, and $M_t$ be the smooth solution of the flow \eqref{flow-VMCF} starting from $M_0$. Then there exists constants $\underline{\kappa}$, $\overline{\kappa}$ depending only on $n$, $\alpha$ and $M_0$ such that the principal curvatures $\kappa_i$ of $M_t$ satisfy:
	\begin{equation}
    \underline{\kappa}\leq \kappa_i\leq \overline{\kappa},\qquad i=1,\dots,n
	\end{equation}
for all time $t\in [0,+\infty)$.
\end{lem}
\proof
By Lemma \ref{innercon}, the center $p_t$ of an inner ball of $\Omega_t$ converges to a fixed point $p$ as $t\to\infty$. Since the inner radius of $\Omega_t$ has a positive lower bound $\rho_-(t)>c_1$, there exists a sufficiently large time $t^{*}$, depending on $c_1$ and hence depending only on $n$, $\alpha$ and $M_0$, such that $d(p_t,p)<{c_1}/{4}$ for $t\geq t^{*}$. Then we have:
\begin{equation}\label{s7.0}
	B_{c_1/4}(p)\subset \Omega_t,\qquad \forall~t\geq t^{*}.
\end{equation}
Applying Proposition \ref{preserve convex} to the time interval $[0,t^*)$ and $[t^*,\infty)$ respectively gives a uniform lower bound for the principal curvatures of $M_t$ for all time $t>0$. In fact, on the time interval $[0,t^*)$, the estimate \eqref{s4.2-0} implies that the principal curvatures $\kappa_i$ of $M_t$ satisfy
\begin{equation}\label{s7.1}
  k_i\geq \Lambda_3^{-1}\Lambda_1^{-\frac{2t^*}{\tau}},\qquad t\in[0,t^*).
\end{equation}
While for time $t\in [t^*,\infty)$, since $B_{c_1/4}(p)\subset \Omega_t$ for all time $t\in [t^*,\infty)$,  without loss of generality, we can view the point $p$ as the origin and project $\Omega_t$ onto $B_1(0)\subset\mathbb{R}^{n+1}$ with respect to the point $p$ for all time $t\geq t^{*}$. By the estimate \eqref{tau1} in the proof of Proposition \ref{preserve convex}, we have
\begin{equation*}
  \mathfrak{r}_{11}(z,t)\leq \Lambda_1\max\left\{\Lambda,\max_{z\in \mathbb{S}^n}\mathfrak{r}_{11}(z,t^*)\right\}
\end{equation*}
for all $t\in [t^*,\infty)$. This together with \eqref{s7.1} and \eqref{s5:W-inv-2} implies that the principal curvatures of $M_t$ are uniformly bounded from below by a positive constant  $\underline{\kappa}$ which depends only on $n$, $\alpha$ and $M_0$.

Since we get the uniform lower bound for the principal curvatures, the uniform upper bound for the Gauss curvature $K$ follows easily from the proof of Proposition \ref{propKupp}. In fact, the upper bound \eqref{s4.2-1} for the terms $(I)$ in the proof of Proposition \ref{propKupp} now has a uniform coefficient $\beta_1$, which is independent of time $t$. Therefore there exists a constant $\overline{\kappa}$, such that $\kappa_i\leq \overline{\kappa}$ for all $i=1,\dots,n$. This completes the proof of Lemma \ref{uni}.
\endproof

It follows from Lemma \ref{uni} that the flow \eqref{flow-VMCF} is uniformly parabolic for all time $t>0$. Then an argument similar to that in the proof of Theorem \ref{long} can be applied to show that all derivatives of curvatures are uniformly bounded on $M_t$ for all $t>0$. This together with Lemma \ref{subcon} implies there exists a sequence of times $t_i\to\infty$, such that $M_{t_i}$ converges smoothly to a geodesic sphere $S_{\rho_{\infty}}(p)$ as $t_i\to\infty$.

The full convergence and the exponential convergence can be obtained by studying the linearization of the flow \eqref{flow-VMCF}. Fix a sufficiently large time $t_i$, we write $M_{t}$ for $t\geq t_i$ as the graph of the radial function $\rho(\cdot,t)$ over $\mathbb{S}^n$ centered at the fixed point $p\in \mathbb{H}^{n+1}$ and the flow equation \eqref{flow-VMCF} is equivalent to the following scalar parabolic PDE
\begin{equation}\label{eq-rho}
	\left\{\begin{aligned}
		\frac{\partial}{\partial t}\rho&=\left(\phi(t)-K^{\alpha}\right)\sqrt{1+{|\bar{\nabla}\rho|^2}/{\sinh^2\rho}},\quad t>t_i,\\
		\rho(\cdot&,t_i)=\rho_{t_i}(\cdot),
	\end{aligned}\right.
\end{equation}
for the radial function $\rho$ over the sphere $\mathbb{S}^n$, where $K$ is expressed as a function of $\rho, \bar{\nabla}\rho$ and $\bar{\nabla}^2\rho$ via the equation \eqref{eq-Gauss}. Since $M_{t_i}$ converges to a geodesic sphere $S_{\rho_\infty}(p)$ as $t_i\to \infty$, we can assume that the oscillation of $\rho(\cdot,{t_i})-\rho_{\infty}$ is sufficiently small by choosing $t_i$ large enough. By a direct computation, the linearized equation of the flow \eqref{eq-rho} about the geodesic sphere $S_{\rho_\infty}(p)$ is given by
\begin{equation}\label{eqdeta2}
	\frac{\partial }{\partial t}\eta=\frac{ \alpha\coth^{\alpha-1}{\rho_{\infty}}}{n\sinh^2\rho_{\infty}}\left(\bar{\Delta}\eta+n\eta-\frac{n}{|\mathbb{S}^n|}\int_{\mathbb{S}^n}{\eta\, d\sigma}\right).
\end{equation}

Since the oscillation of $\rho(\cdot,{t_i})-\rho_{\infty}$ is already sufficiently small, it follows exactly as in \cite{Cab-Miq2007}, using \cite{Esc98}, that the solution $\rho(\cdot,t)$ of \eqref{eq-rho} starting at $\rho(\cdot,{t_i})$ exists for all time and converges exponentially to a constant $\rho_{\infty}$. This means that the hypersurface $\overline{M}_t=$ graph $\rho(\cdot,t)$ solves \eqref{flow-VMCF} with initial condition $M_{t_i}$ and by uniqueness $\overline{M}_t$ coincides with $M_t$ for $t\geq t_i$, and hence the solution $M_t$ of \eqref{flow-VMCF} with initial condition $M_0$ converges exponentially as $t\to\infty$ to the geodesic sphere $S_{\rho_\infty}(p)$. This completes the proof of Theorem \ref{theo}.

\subsection{Proof of Corollary \ref{coro}}\label{sec6}
Finally, we give the proof of Corollary \ref{coro} using the monotonicity of $\mathcal{A}_{n-1}(\Omega_t)$ and the convergence result of the flow \eqref{flow-VMCF}. Firstly, if $\Omega$ is convex, we evolve the boundary $M=\partial\Omega$ by the flow $\eqref{flow-VMCF}$. Then the inequality \eqref{eqA-F} follows from Theorem \ref{theo} and the monotonicity in Lemma \ref{lemmono} immediately. If equality holds in \eqref{eqA-F} for such $\Omega$, then equality also holds in \eqref{eqWmo} for all time $t$ which means that $\Omega
$ and $\Omega_t$ are all geodesic balls.

In general, for weakly convex $\Omega$ we can approximate $\Omega$ by a family of convex domains $\Omega_\varepsilon$ as $\varepsilon\to 0$. In fact, if we project the domain $\Omega$ into $B_1(0)\subset \mathbb{R}^{n+1}$ as in $\S$\ref{sec2}, the equation \eqref{eq-hij} implies that the image $\hat{\Omega}$ of the projection is also weakly convex in  $\mathbb{R}^{n+1}$. Hence we can approximate $\hat{\Omega}$ by convex domains (e.g., by using mean curvature flow). Since the projection is a diffeomorphism, we find a family of convex domains $\Omega_{\varepsilon}$ that approximate $\Omega$ as $\varepsilon\to 0$. Then the inequality \eqref{eqA-F} for $\Omega$ follows from the one for $\Omega_\varepsilon$ by letting $\varepsilon\to 0$.

To prove the equality case for weakly convex $\Omega$, we employ an argument previously used in \cite{G-L2009}. Suppose that $\Omega$ is a weakly convex domain which attains the equality of \eqref{eqA-F}. Let $M_{+}=\{x\in M=\partial\Omega, K>0\}$. Since there exists at least one point $p$ on a closed hypersurface in $\mathbb{H}^{n+1}$ such that all the principal curvatures are strictly larger than 1 at $p$, the subset $M_{+}$ is open and nonempty. We claim that $M_{+}$ is closed as well. In fact, pick any $\eta\in C_0^2(M_{+})$ compactly supported in $M_{+}$, let $M_\varepsilon$ be a smooth family of variational hypersurfaces generated by the vector field $V=\eta\nu$. Let $\Omega_\varepsilon$ be the domain enclosed by $M_\varepsilon$. It is easy to show that $M_\varepsilon$ is weakly convex when $|\varepsilon|$ is small enough. Hence
\begin{equation*}
	\mathcal{A}_{n-1}(\Omega_\varepsilon)\geq \psi_n\left(|\Omega_\varepsilon|\right)
\end{equation*}
holds for sufficiently small $|\varepsilon|$ and with equality holding at $\varepsilon=0$. Thus
\begin{align*}
	0=&\frac{d}{d\varepsilon}\bigg|_{\varepsilon=0}\Big(\mathcal{A}_{n-1}(\Omega_\varepsilon)- \psi_n\left(|\Omega_\varepsilon|\right)\Big)=n\int_{M}{\left(K-\psi'_n(|\Omega|)\right)\eta}d\mu.
\end{align*}
Since $\eta\in C_0^2(M_{+})$ is arbitrary, we have $K=\psi'_n(|\Omega|)$ everywhere on $M_{+}$. As this is a closed condition, we conclude that $M_{+}$ is closed.  Therefore $M=M_{+}$ and so $\Omega$ is a convex domain. Then by the equality case for convex domain, we conclude that $\Omega$ is a geodesic ball.


\end{document}